\documentclass[12pt]{amsart}
\usepackage{geometry}
\usepackage{amsmath}
\usepackage{amssymb}
\usepackage{amsthm}
\usepackage{bm}
\usepackage{mathtools}
\usepackage{mathrsfs}
\usepackage{graphicx}
\usepackage[dvipsnames]{xcolor}

\newcommand{\N}{\mathbb{N}}

\newcommand{\F}{\mathcal{F}}
\newtheorem{proposition}{Proposition}
\newtheorem{theorem}{Theorem}
\newtheorem{corollary}{Corollary}
\newtheorem{remark}{Remark}

\DeclareMathOperator{\Tr}{Tr}

\title[Convergence speed for the average density of eigenfunctions]{Convergence speed for the average density of eigenfunctions for singular Riemannian manifolds}
\author{Charlotte Dietze}
\address[Charlotte Dietze]{Sorbonne Université, CNRS, Laboratoire Jacques-Louis Lions,
4 place Jussieu,
75005 Paris,
France
}
 \email{Charlotte.Dietze@sorbonne-universite.fr}

\begin{document}
\maketitle
\begin{abstract}
We consider a class of singular Riemannian metrics on a compact Riemannian manifold with boundary and the eigenfunctions of the corresponding Laplace-Beltrami operator. In our setting, the average density of eigenfunctions with eigenvalue less than $\lambda$ converges weakly to the uniform normalised measure on the boundary as $\lambda\to\infty$. In this work, we show a quantitative estimate on the speed of this convergence in the Wasserstein-sense in the transverse coordinate to the boundary. 
\end{abstract}
\section{Introduction}
We consider a smooth compact Riemannian manifold $X$ with boundary of topological dimension $n+1$, $n\in\N$. We equip $X$ with a Riemannian metric $g$ that is in a certain class of singular Riemannian metrics $g$ with a power-type singularity at the boundary. More precisely, we identify a tubular neighbourhood of the boundary of $X$ by $[0,1]_x\times M_y=:X_1$, where $M$ is an $n$-dimensional Riemannian manifold that is isomorphic to $\partial X$. The class of metrics $g$, which we consider, are smooth in the interior of $X$ and for a parameter $\beta\in[2/n,\infty)$, they satisfy on the tubular neighbourhood $X_1$ of $\partial X$, which is identified with $[0,1]\times M$,
\begin{equation}\label{eq:g}
g=dx^2+x^{-\beta} g_1(x),
\end{equation}
where $(g_1(x))_{x\in[0,1]}$ is a continuous family of smooth, non-degenerate metrics on $M$. We denote the corresponding (non-negative) Laplace-Beltrami operator with Dirichlet boundary conditions on the boundary of $X$ by $\Delta$, and the Riemannian volume measure on $X$ with respect to $g$ is denoted by $dv_g$. For $\beta=2$, the Laplace-Beltrami operator corresponds to a Grushin-type operator plus a potential of the form $x^{-2}$ after a suitable change of functions. 

\medskip

In this paper, we give a quantitative estimate of how fast the average density of eigenfunctions of $\Delta$ accumulates at the boundary of $X$. 
\subsection*{Statement of the main theorem} 
In \cite{yves}, it was shown that the average density of eigenfunctions of $\Delta$ converges weakly to the uniform distribution on the boundary of $X$. In this paper, we go one step further and give a quantitative estimate of the speed of this convergence in the Wasserstein-$p$-sense in the $x$-coordinate, see Corollary \ref{co:Wasserstein} below.

\medskip

More precisely, we consider the average density $F_\lambda$ of eigenfunctions of $\Delta$ with eigenvalue below $\lambda$, which is normalised such that $\int_X F_\lambda(z)\, dv_g(z)=1$, see \eqref{eq:F_lambda_def} below for the definition of $F_\lambda$. We show below in Theorem \ref{th:main} for $p=1$ and Theorem \ref{th:moments} for $p\ge2$ that 
\begin{equation}\label{eq:Wasserstein_x_simplified_into}
\int_{X} \min\{x^p,1\}F_\lambda(x,y)\,dv_g(x,y)\le C A_\lambda,
\end{equation}
where $A_\lambda$ is an explicit convergence rate depending on $p$, $\beta$, $n$ and $\lambda$ with $A_\lambda\to0$ as $\lambda\to\infty$, see 
\eqref{eq:A_lambda_th_main_cases} and \eqref{eq:A_lambda_th_moments_cases}
for the precise definitions of $A_\lambda$. Moreover, for $p\ge2$ and $\frac{2}{n}\le\beta\le\frac{6}{n}$, we \emph{cannot} hope for a much better convergence rate $A_\lambda$, see Remark \ref{re:optimality} below.

\medskip

\subsection*{General literature}
In this paper, we consider a Laplace-Beltrami operator of a singular Riemannian metric, which corresponds for $\beta=2$ to a Grushin-type operator plus a potential. There is a vast literature on the spectral theory for such Laplace-Beltrami operators of singular Riemannian metrics and sub-Riemannian geometry, and in particular the Grushin operator. Many of these works show Weyl asymptotics, or asymptotics for the corresponding heat kernel. 

\medskip

The small time asymptotics for sub-Riemannian heat kernels were studied in \cite{colin2018spectral,colin2021small,colin2022spectral,chang2015heat}. In \cite{menikoff1978eigenvalues}, the authors consider classical pseudo-differential operators on smooth manifolds with a principal symbol that vanishes  exactly to second order on a smooth symplectic submanifold, which corresponds to the Grushin case. In this setting, they prove Weyl asymptotics. Another paper considering Weyl asymptotics in the Grushin case is \cite{boscain}, where the eigenfunctions and eigenvalues are computed explicitly, which also allows to deduce Weyl asymptotics. There are many more references dealing with the Grushin case, such as \cite{abatangelo2025solutions, lamberti2021shape}.  Moreover, we would like to mention the work \cite{chitour}, where Weyl asymptotics were shown for a large class of singular Riemannian metrics, including the Grushin case. 

\medskip

For the model, which we consider in the present paper, Weyl asymptotics were proved in \cite{yves}, and we would also like to refer to the earlier works \cite{metivier, vulis}. In \cite{yves}, we also showed that if $\beta\ge 2/n$, most of the eigenfunctions with eigenvalue less than $\lambda$ accumulate near the boundary as $\lambda\to\infty$. Here $\beta=2/n$ is a critical value. In \cite{larry} for the case $\beta>2/n$ and in \cite{charlotte} for the critical case $\beta=2/n$, we identified the typical length-scale of where eigenfunctions are localised near the boundary.  

\subsection*{Comparison of the main result \eqref{eq:Wasserstein_x_simplified_into} with earlier works}
In \cite{larry} for the case $\beta>2/n$ and in \cite{charlotte} for $\beta=2/n$, we gave a precise description of at which scale near the boundary the average density of eigenfunctions of $\Delta$ lives. The results in those works contain information of the following form: For any $\epsilon>0$, there exists an explicit $L_\lambda\le 1$ depending on $\beta$, $n$ and $\lambda$ and with $L_\lambda\to0$ as $\lambda\to\infty$ such that
\begin{equation}\label{eq:Larry_Charlotte}
\int_{X_{L_\lambda}} F_\lambda(z)\,dv_g(z)\ge 1-\epsilon, 
\end{equation}
where we use for any $a\in[0,1]$ the notation $X_a:=[0,a]\times M$, which we identify with the corresponding tubular neighbourhood of the boundary of $X$. Note that estimates of the form \eqref{eq:Larry_Charlotte} are insufficient to prove \eqref{eq:Wasserstein_x_simplified_into}. In \cite{larry}, we also showed that in the case $\beta>2/n$ and for fixed smooth and \emph{bounded} functions $V$, the quantity
\begin{equation}
\int_{X_1} V(\sqrt\lambda x,y)F_\lambda(x,y)\,dv_g(x,y)
\end{equation}
converges to a finite number as $\lambda\to\infty$. So in particular, due to the boundedness condition on $V$, we \emph{cannot} take $V(x,y)=x^p$ in \cite{larry} to deduce estimates of the form \eqref{eq:Wasserstein_x_simplified_into}.

\medskip

\subsection*{Literature on the convergence speed for the average density of eigenfunctions} 
To our knowledge, there are only very few results in the literature that give a quantitative estimate on the convergence speed of the average density of eigenfunctions, for instance in some Wasserstein-sense. In the regular case of a compact smooth Riemannian manifold $X$ without boundary, equipped with a non-singular metric, we refer to Hörmander \cite[Theorem 17.5.7]{hoermander3}, which shows that the average density of eigenfunctions $F_\lambda$ converges in the $L^\infty$-norm to the uniform measure on $X$ with a convergence speed of $\lambda^{-1}$. From \cite[Theorem 17.5.7]{hoermander3}, one can also show a corresponding statement in the case of a smooth Riemannian manifold with boundary, for instance in a Wasserstein-$1$-sense.

\medskip

\subsection*{Notation}
We denote the number of eigenvalues of $\Delta$ with eigenvalues less than $\lambda$ by $N(\lambda)$. Here we count the eigenvalues with multiplicity. 

The average density $F_\lambda$ of eigenfunctions of $\Delta$ with eigenvalue below $\lambda$ is defined as follows: for any orthonormal basis $\{\Psi\}$ of the span of all eigenfunctions of $\Delta$ with eigenvalues less than $\lambda$, we define 
\begin{equation}\label{eq:F_lambda_def}
F_{\lambda}(z) := \frac{1}{N(\lambda)} \sum_{\substack{\Psi \text{ } L^{2}(X,dv_{g})\text{-normalized} \\ \text{eigenfunction of } \Delta \\ \text{with eigenvalue } < \lambda}} |\Psi(z)|^2, \quad z \in X.
\end{equation}
In particular, we have
\begin{equation}\label{eq:F_lambda_integral}
\int_{X} F_{\lambda}(z) \, dv_{g}(z) = 1.
\end{equation}
Note that the definition of $F_\lambda$ is independent of the orthonormal basis $\{\Psi\}$ chosen. To see this, note that for any continuous and bounded function $h$ on $X$, we have
\begin{equation}
\int_{X} h(z)F_{\lambda}(z) \, dv_{g}(z)=\Tr_{L^{2}(X,dv_{g})}(h1_{\Delta<\lambda}). 
\end{equation}

\medskip

Below, we will use $C>0$ as a constant that can change from line to line.  

\medskip

\subsection*{Detailed statement of the main result}
Let us now state the main theorem.

\begin{theorem}\label{th:main}
Let $\beta\in[2/n,\infty)$ and let $F_\lambda$ be the normalised density of eigenfunctions of $\Delta$ on $X$ with eigenvalues $<\lambda$, defined in \eqref{eq:F_lambda_def}. Define
\begin{equation}\label{eq:A_lambda_th_main_cases}
A_\lambda:=\begin{cases}
\lambda^{-\frac{1}{2}+\frac{1}{n\beta}} &\text{if } \beta>\frac{2}{n} \\[2mm]
\frac{\log(\log(\lambda))}{\log(\lambda)}&\text{if } \beta=\frac{2}{n}
\end{cases}
\end{equation}
%
Then there exists a constant $C>0$ and $\lambda_0>0$ such that for all $\lambda\ge\lambda_0$, 
\begin{equation}\label{eq:X_ohne_X_A_lambda_into_th}
\int_{X\setminus X_{A_\lambda}} F_\lambda(x,y)\,dv_g(x,y)\le C A_\lambda. 
\end{equation}
In particular, we also have for a possibly different constant $C>0$ that 
\begin{equation}\label{eq:Wasserstein_x_simplified_into_th}
\int_{X} \min\{x,1\}F_\lambda(x,y)\,dv_g(x,y)\le C A_\lambda. 
\end{equation}
\end{theorem}

\begin{remark}\label{re:xF_von_Indikator_beim_Rand}
To see how \eqref{eq:Wasserstein_x_simplified_into_th} follows from \eqref{eq:X_ohne_X_A_lambda_into_th}, note that 
\begin{equation}
\int_{X} \min\{x,1\}F_\lambda(x,y)\,dv_g(x,y)\le \int_{X\setminus X_{A_\lambda}} F_\lambda(x,y)\,dv_g(x,y)+A_\lambda\le (C+1)A_\lambda.
\end{equation}
\end{remark}

We can think of \eqref{eq:Wasserstein_x_simplified_into_th} as an estimate of the first moment of $F_\lambda \, dv_g$. The following theorem provides an estimate on $p$-th moments for $p\ge 2$. 

\begin{theorem}\label{th:moments}
Let $\beta\in[2/n,\infty)$, $p\ge2$ and let $F_\lambda$ be the normalised density of eigenfunctions of $\Delta$ on $X$ with eigenvalues $<\lambda$, defined in \eqref{eq:F_lambda_def}. Define
\begin{equation}\label{eq:A_lambda_th_moments_cases}
A_\lambda:=\begin{cases}
\frac{1}{\log(\lambda)}&\text{if }\beta=\frac{2}{n}\\[2mm]
\frac{\log(\lambda)}{\lambda} &\text{if } \beta=\frac{6}{n},\, p=2 \\[2mm]
\lambda^{\max\left\{\frac{1}{2}-\frac{n\beta}{4},-1\right\}}&\text{else } \end{cases}
\end{equation}
%
Then there exists a constant $C>0$ and $\lambda_0>0$ such that for all $\lambda\ge\lambda_0$,  
\begin{equation}\label{eq:moments theorem}
\int_{X} \min\{x^p,1\} F_\lambda(x,y)\,dv_g(x,y)\le C A_\lambda. 
\end{equation}
\end{theorem}

Both in \eqref{eq:X_ohne_X_A_lambda_into_th} and for $p\ge2$ and $\frac{2}{n}\le\beta\le\frac{6}{n}$ for \eqref{eq:moments theorem}, we \emph{cannot} hope for a much better estimate. We make this precise in Remark \ref{re:optimality} below. 

\subsection*{Estimates for Wasserstein-$p$-distances in the $x$-coordinate}
The estimates \eqref{eq:Wasserstein_x_simplified_into_th} and \eqref{eq:moments theorem} imply estimates on different Wasserstein distances of the probability measure $F_\lambda\,dv_g$ to the uniform probability measure on the boundary $\partial X$ with respect to $G=g_1(x=0)$ after projection on the $x$-coordinate. 

\bigskip

More precisely, denote by $\mu_\lambda$ the probability measure on $[0,1]$ that satisfies for all $a\in(0,1)$
\begin{equation}\label{eq:mu_def1}
\int_{[0,a]}\,d\mu_\lambda=\int_{X_a}F_\lambda\,dv_g\,. 
\end{equation}
In particular, since $\mu_\lambda$ is a probability measure and $F_\lambda$ is absolutely continuous with respect to the Riemannian measure on $X$, we have
\begin{equation}\label{eq:mu_def2}
\mu_\lambda(\{1\})=\int_{X\setminus X_1}F_\lambda\,dv_g\,. 
\end{equation}
We will compare $\mu_\lambda$ to the delta-distribution $\delta_0$ at zero, which is also a probability measure on $[0,1]$. Here we identify $\delta_0$ with the projection on the $x$-coordinate of the uniform probability measure on the boundary of $X$. 

\bigskip

Recall for any $p\ge1$ the definition of the Wasserstein-$p$-distance $W_p(\mu, \nu)$ of two probability measures $\mu$ and $\nu$ on $[0,1]$. We denote by $\Gamma(\mu,\nu)$ the space of all couplings of $\mu$ and $\nu$, that is, measures on $[0,1]\times [0,1]$ with marginals $\mu$ in the first coordinate, and $\nu$ in the second coordinate. 
Then the Wasserstein-$p$-distance of $\mu$ and $\nu$ is defined by
\begin{equation}
W_p(\mu, \nu):=\left(\inf_{\gamma\in \Gamma(\mu,\nu)}\int_{[0,1]\times [0,1]}|x-\tilde x|^p\, d\gamma(z,\tilde z)	\right)^{\frac 1 p}.
\end{equation} 

\bigskip

Thus, applying this definition to $\mu_\lambda$ and $\delta_0$, we find that 
\begin{equation}\label{eq:Wasserstein_distance_in_x}
W_p(\mu_\lambda, \delta_0)^p=\int_{[0,1)}x^p\,d\mu_\lambda+\mu_\lambda(\{1\})=\int_{X}\min\{x^p,1\}F_\lambda(x,y)\,dv_g(x,y). 
\end{equation}

\begin{corollary}[Estimate on the convergence speed in the Wasserstein-$p$-distance]\label{co:Wasserstein}
Denote by $\mu_\lambda$ the probability measure on $[0,1]$ given by the projection of the average density of eigenfunctions on the $x$-coordinate, as defined in \eqref{eq:mu_def1} and \eqref{eq:mu_def2}, and denote by $\delta_0$ the delta-distribution in zero, seen as a probability measure on $[0,1]$. Let $\beta\ge2/n$. Then for $p=1$ or any $p\ge2$, there exists a constant $C>0$ such that for all $\lambda$ sufficiently large, we have the following quantitative estimate for the distance of $\mu_\lambda$ and $\delta_0$ in the Wasserstein-$p$-distance
\begin{equation}\label{eq:Wasserstein_co_est}
W_p(\mu_\lambda, \delta_0)\le C A_\lambda^{1/p}
\end{equation}
for $A_\lambda$ defined in \eqref{eq:A_lambda_th_main_cases} for $p=1$ and $A_\lambda$ defined in \eqref{eq:A_lambda_th_moments_cases} for $p\ge2$.
\end{corollary}
\begin{proof}
It suffices to combine \eqref{eq:Wasserstein_distance_in_x} and the statements of Theorem \ref{th:main} and Theorem \ref{th:moments}. 
\end{proof}
Let us remark that one could also obtain estimates similar estimates to \eqref{eq:Wasserstein_co_est} for $p\in(1,2)$ by interpolating between the result from Theorem \ref{th:main} for $p=1$ and Theorem \ref{th:moments} for $p=2$. 

\subsection*{The quasi-isometric Riemannian metric $\tilde g$ and its Laplace-Beltrami operator $\tilde\Delta$}

A key ingredient of the proof of Theorem \ref{th:main} is the comparison of the Laplace-Beltrami operator $\Delta$ of $(X,g)$ with the Laplace-Beltrami operator $\tilde\Delta$ of $(X,\tilde g)$ for a singular Riemannian metric $\tilde g$ with nicer properties. For fixed $\epsilon\in(0,1)$, $\tilde g$ has a separable structure on $X_\epsilon$, and  $\tilde g$ is quasi-isometric to $g$ up to a constant $(1+\delta)$, namely in the matrix-sense, we have
\begin{equation}\label{eq:g_g_tilde_est}
\frac{1}{1 + \delta}\tilde g\le g\le (1 + \delta) \tilde g. 
\end{equation}
Here $\delta>0$ depends on $\epsilon$ and $\delta\to0$ as $\epsilon\to0$. Moreover, by possibly replacing $\delta$ by a small constant times $\delta$, we can choose $\delta\in(0,1)$ such that we also have
\begin{equation}\label{eq:volume_quasi_isometry}
\frac{1}{1 + \delta}dv_{\tilde g}\le dv_{ g}\le (1 + \delta) dv_{\tilde g}
\end{equation}
and 
\begin{equation}
\frac{1}{1 + \delta} \langle f, \tilde\Delta f\rangle_{L^2(X, dv_{\tilde{g}})}
\le \langle f, \Delta f\rangle_{L^2(X, dv_{g})}
\le (1 + \delta)\langle f, \tilde\Delta f\rangle_{L^2(X, dv_{\tilde{g}})} \quad \text{for all } f\in C_c^\infty(X).
\end{equation}

\medskip

For given $\epsilon\in(0,1/3)$, we define the metric $\tilde g$ as follows: Define the metric $G$ on $M$ by $G:=g_1(x=0)$. Here the $(g_1(x))_{x\in[0,1]}$ was the continuous family of smooth, non-degenerate metrics on $M$ from \eqref{eq:g}. Recall that $g$ is such that on $X_1\cong [0,1]\times M$, 
\begin{equation}
g=dx^2+x^{-\beta} g_1(x). 
\end{equation}
We define $\tilde g$ as equal to $g$ on $X\setminus X_{1}$. For a fixed smooth function $\eta:[0,1]\to[0,1]$ with  $\eta \equiv 1$ for $x \leq 2\epsilon$, $\eta \equiv 0$ for $x \geq 3\epsilon$, we define $\tilde g$ on $X_1\cong [0,1]\times M$ by 
\begin{equation}\label{eq:tilde_g_def}
\tilde{g} = dx^2 + x^{-\beta} \left( \eta(x) G +(1 -\eta(x)) g_1(x) \right).
\end{equation}
The estimate \eqref{eq:g_g_tilde_est} with $\delta>0$ small for $\epsilon>0$ small follows from the continuity of the family $(g_1(x))_{x\in[0,1]}$. For the purpose of this paper, it is not necessary to have $\delta$ small. Nevertheless, we keep track of $\delta$, in particular in the proof of Proposition \ref{prop:localisation} below.

\medskip

In the proof of Proposition \ref{prop:localisation} below, we will use the notations $\nabla$ and $\tilde\nabla$. For two functions $F_1$ and $F_2$ on $X$, the expression $\nabla F_1\cdot\nabla F_2$ in local coordinates denotes  the weighted scalar product $\langle \nabla F_1, g^{-1} \nabla F_2 \rangle$ in $\mathbb{R}^{n+1}$, where $g^{-1}$ is identified as a matrix. In particular, 
\begin{equation}
\int_X \nabla F_1\cdot \nabla F_2\, dv_g =\langle F_1,\Delta F_2\rangle_{L^2(X, dv_{g})}.\end{equation}
We have the same definition for $\tilde \nabla$, replacing $g$ by $\tilde g$ everywhere, so 
\begin{equation}
\int_X \tilde\nabla F_1\cdot \tilde\nabla F_2\, dv_{\tilde g} =\langle F_1,\tilde \Delta F_2\rangle_{L^2(X, dv_{\tilde g})}.
\end{equation}

\medskip

After a change of function and measure, see \cite[Section 3.8]{PhD}, the corresponding Laplace-Beltrami operator $\tilde\Delta$ can be written as
\begin{equation}\label{eq:tilde_Delta}
\tilde\Delta = -\partial_x^2+\frac{C_\beta}{x^2}+x^\beta\Delta_M, \quad \text{on } L^2\left([0,\epsilon]\times M, dx\, dv_G(y)\right),\quad C_\beta:=\frac{\beta n}{4}\left(1+\frac{\beta n}{4}\right),
\end{equation}
where $\Delta_M$ denotes the Laplace-Beltrami operator on $M$ with respect to $G$. Thus, $\Delta$ and $\tilde\Delta$ can be seen as a $\beta$-Grushin-type operator plus a potential, where the case $\beta=2$ corresponds to the classical Grushin case. 

\bigskip

Next, let us explain the proof stategy for Theorem \ref{th:main} and for Theorem \ref{th:moments}.

\subsection*{Proof strategy for Theorem \ref{th:main}}
The proof of \eqref{eq:X_ohne_X_A_lambda_into_th} is structured as follows:

{\bf Step 1: Localisation estimate involving a trace of $\tilde\Delta$. }
In Proposition \ref{prop:localisation}(ii) below, we give an estimate for $F_\lambda$, the average density of eigenfunctions of $\Delta$, in domains of the form $X\setminus X_{A_\lambda}$. The estimate only involves a trace in terms of $\tilde \Delta$ and a localisation error. 

{\bf Step 2: Estimate for the number of eigenvalues of $\tilde\Delta$. }
We estimate the trace in terms of $\tilde \Delta$ from step 1 by providing a corresponding estimate on the number of eigenvalues in Proposition \ref{prop:number_eigenvalues}(i),(ii) below.

{\bf Step 3: Estimate of the localisation error and conclusion. }
In order to conclude, it remains to control the localisation error from step. In the critical case $\beta=2/n$, this is straightforwardly possible. In the supercritical case $\beta>2/n$, we use a bootstrap argument to obtain an optimal estimate.

\subsection*{Proof strategy for Theorem \ref{th:moments}}
Let us now explain the proof strategy of Theorem \ref{th:moments}.

{\bf Step 1: Localisation estimate involving a trace of $\tilde\Delta$. }
The first step is very similar to the first step of Theorem \ref{th:main}. Namely, in Proposition \ref{prop:localisation}(i), we express the quantity
\begin{equation}
\int_{X} \min\{x^p,1\} F_\lambda(x,y)\,dv_g(x,y),
\end{equation}
which we want to estimate, in terms of a trace involving the operator $\tilde \Delta$ and a localisation error. 

{\bf Step 2: Estimate for the number of eigenvalues of $\tilde\Delta$. }
In Proposition \ref{prop:number_eigenvalues}(iii) and (iv), we provide estimates for eigenvalues of $\tilde\Delta$ on domains of the form $X_{2B_\lambda} \setminus X_{B_\lambda}$, depending on the value of $B_\lambda$. 

{\bf Step 3: Conclusion using a dyadic decomposition. }
Finally, to conclude, we decompose $X$ into $X_{\frac{1}{2}\lambda^{-1/2}}$, $X\setminus X_\epsilon$, and a union of sets of the form $X_{2B_\lambda} \setminus X_{B_\lambda}$. This allows us to estimate $x^p$ from above and below on each of those sets $X_{2B_\lambda} \setminus X_{B_\lambda}$. Using the estimates from Proposition \ref{prop:number_eigenvalues}(iii) and (iv), we can estimate both the localisation error and the main term with a trace coming from Proposition \ref{prop:localisation}(i), which gives the desired result.

\subsection*{Possible Generalisations} The methods used in this paper are relatively robust. Using the strategy presented here, it should be possible to treat the Laplace-Beltrami operator of other classes of singular Riemannian metrics, or sub-Riemannian Laplacians. The main ingredient that is needed to make the proof work in other settings is the analogue of a localised estimate of the number of eigenvalues of a model operator, as in Proposition \ref{prop:number_eigenvalues}. 

\subsection*{Structure of the paper}
In Section \ref{s:localisation}, we prove Proposition \ref{prop:localisation}. In Section \ref{s:number_eigenvalues}, we prove Proposition \ref{prop:number_eigenvalues}. We conclude the proof of Theorem \ref{th:main} in Section \ref{s:conclusion}. In Section \ref{s:conclusion_moments}, we prove Theorem \ref{th:moments}.

\subsection*{Acknowledgements}
The author would like to thank Yves Colin de Verdière and Emmanuel Trélat for their support, very helpful discussions and suggestions. The author would like to thank the Isaac Newton Institute for Mathematical Sciences, Cambridge, for support and hospitality during the programme Geometric spectral theory and applications, where work on this paper was undertaken. This work was supported by EPSRC grant EP/Z000580/1. She also acknowledges the support by the European Research Council via ERC CoG RAMBAS, Project No. 101044249.

\section{Localisation estimates involving a trace of $\tilde\Delta$}\label{s:localisation}
In this section, we prove Proposition \ref{prop:localisation}. In the proof of Theorem \ref{th:moments}, we will apply Proposition \ref{prop:localisation}(i) for functions $\chi(x,y)=\min\{x^{p/2},1\}$. In the proof of Theorem \ref{th:main}, we will use Proposition \ref{prop:localisation}(ii) for a smooth cut-off functions at scale $A_\lambda$.
\begin{proposition}\label{prop:localisation}
Let $\tilde g$ be as in \eqref{eq:tilde_g_def} an approximation of $g$ with the corresponding $\epsilon>0$ and $\delta>0$, and denote by $\tilde\Delta$ the Laplace-Beltrami operator corresponding to $\tilde g$, see \eqref{eq:tilde_Delta}. 
\begin{itemize}
\item[(i)] Let $\chi:X\to [0, 1]$ be a function that is weakly differentiable. Then we have
\begin{equation}
\begin{split}
\sum_{\substack{\Psi \text{ eigenfunction} \\ \text{of } \Delta \text{ with eigenvalue } < \lambda}} \int_{X }\chi^2 |\Psi|^2 \, dv_g \leq \frac{1}{\lambda} \Tr_{L^{2}(X, dv_{\tilde g})} \left( \chi(2(1+\delta)^2 \lambda - \tilde{\Delta})\chi\right)_{+} \\
\quad\quad + \frac{1}{\lambda} \sum_{\substack{\Psi \text{ eigenfunction} \\ \text{of } \Delta \text{ with eigenvalue } < \lambda}} \int_{X} |\nabla\chi|^2 |\Psi|^2 \, dv_g.
\end{split}
\end{equation}
\item[(ii)]
Let $0 < a < b < \epsilon$ and let $\chi : [0, \epsilon] \to [0, 1]$ 
be weakly differentiable 
with $\chi(x) = 0$ for $x \leq a$ and $\chi(x) = 1$ for $x \geq b$. Then we have
\begin{equation}
\begin{split}
\sum_{\substack{\Psi \text{ eigenfunction} \\ \text{of } \Delta \text{ with eigenvalue } < \lambda}} \int_{X \setminus X_b} |\Psi|^2 \, dv_g \leq \frac{1}{\lambda} \Tr_{L^{2}(X \setminus X_a, dv_{\tilde g})} \left( (2(1+\delta)^2 \lambda - \tilde{\Delta})_{+} \right)\\
\quad\quad + \frac{1}{\lambda} \sum_{\substack{\Psi \text{ eigenfunction} \\ \text{of } \Delta \text{ with eigenvalue } < \lambda}} \int_{X} |\chi'|^2 |\Psi|^2 \, dv_g,
\end{split}
\end{equation}
where we have identified $\chi'(z)$ for $z=(x,y)\in[0,\epsilon]\times M$ with $\chi'(x)$, and $0$ else. Moreover, we take Neumann boundary conditions for $\tilde\Delta$ on $\partial X_a$. 
\end{itemize}

\end{proposition}
\begin{proof}
Let $\Psi$ be an eigenfunction of $\Delta$ with eigenvalue $\mu \in [0, \lambda]$. Then
\begin{equation}\label{eq:eigenfunction}
\Psi (2 \lambda - \Delta) \Psi = (2 \lambda - \mu) |\Psi|^2 \geq \lambda |\Psi|^2.
\end{equation}

By \eqref{eq:eigenfunction},
\begin{equation}\label{eq:proof_step1}
\int_{X } \chi^2|\Psi|^2 \, dv_g \leq \frac{1}{\lambda} \int_{X} \chi^2 \Psi (2 \lambda - \Delta) \Psi \, dv_g = \frac{1}{\lambda} \left( 2 \lambda \int_{X} |\chi \Psi|^2 \, dv_g - \int_{X} \nabla (\chi^2 \Psi) \cdot \nabla \Psi \, dv_g \right).
\end{equation}

Here, $\nabla F_1 \cdot \nabla F_2$ in local coordinates corresponds to taking the weighted scalar product $\langle \nabla F_1, g^{-1} \nabla F_2 \rangle$ in $\mathbb{R}^n$, where $g^{-1}$ is identified as a matrix. By the product rule, we have
\begin{equation}\label{eq:gradient_product}
\begin{split}
\int_{X} \nabla (\chi^2 \Psi) \cdot \nabla \Psi \, dv_g = \int_{X} \chi \nabla (\chi \Psi) \cdot \nabla \Psi \, dv_g + \int_{X} \chi \Psi (\nabla \chi) \cdot \nabla \Psi \, dv_g\\
= \int_{X} |\nabla (\chi \Psi)|^2 \, dv_g - \int_{X} \Psi (\nabla \chi) \cdot \nabla (\chi \Psi) \, dv_g + \int_{X} \chi \Psi (\nabla \chi) \cdot \nabla \Psi \, dv_g.
\end{split}
\end{equation}
Since $\nabla (\chi \Psi)= \Psi\nabla\chi+\chi\nabla\Psi$, 
we get
\begin{equation}
\int_{X} \nabla (\chi^2 \Psi) \cdot \nabla \Psi \, dv_g=\int_{X} |\nabla (\chi \Psi)|^2 \, dv_g - \int_{X} |\nabla\chi|^2 |\Psi|^2 \, dv_g.
\end{equation}
Note that this corresponds to the IMS formula on Riemannian manifolds. Thus,
\begin{equation}\label{eq:ims_formula}
\int_{X} \chi^2|\Psi|^2 \, dv_g \leq \frac{1}{\lambda} \left( 2 \lambda \int_{X} |\chi \Psi|^2 \, dv_g - \int_{X} |\nabla (\chi \Psi)|^2 \, dv_g + \int_{X} |\nabla\chi|^2 |\Psi|^2 \, dv_g \right).
\end{equation}

Using the quasi-isometry of $g$ and $\tilde{g}$, we have
\begin{equation}\label{eq:chi_Psi_ohne_Tr}
\int_{X} \chi^2 |\Psi|^2 \, dv_g \leq \frac{1}{\lambda} (1 + \delta) \left( 2 \lambda \int_{X} |\chi \Psi|^2 \, dv_{\tilde{g}} - \frac{1}{(1 + \delta)^2} \int_{X} |\tilde{\nabla} (\chi \Psi)|^2 \, dv_{\tilde{g}} \right) + \frac{1}{\lambda} \int_{X} |\nabla\chi|^2 |\Psi|^2 \, dv_g.
\end{equation} 
Later, we will take the sum over an orthonormal family in $L^2(X, dv_g)$ of eigenfunctions $\Psi$ of $\Delta$. 
From this point on, the proof of $(i)$ and $(ii)$ is slightly different.

{\bf End of the proof of $(i)$. }
Using the bra-ket notation, we claim that
\begin{equation}\label{eq:projection_estimate}
0 \leq \sum_{\Psi} | \Psi \rangle \langle  \Psi| \leq (1 + \delta)  \quad \text{on } L^2(X, dv_{\tilde{g}}),
\end{equation}
that is, the sum of the orthogonal projections onto $\Psi$ in $L^2(X, dv_{\tilde{g}})$ is less than or equal in an operator sense to $1+\delta$ times the identity operator. Note that here, the orthogonal projection onto $\Psi$ in $L^2(X, dv_{\tilde{g}})$ is different from the orthogonal projection onto $\Psi$ in $L^2(X, dv_{g})$. 
To see \eqref{eq:projection_estimate}, let $\Phi \in L^2(X, dv_{\tilde{g}})$. Then
\begin{equation}
\left\langle \Phi, \sum_{\Psi} | \Psi \rangle \langle \Psi | \Phi  \right\rangle_{L^2(X, dv_{\tilde{g}})} = \sum_{\Psi} |\langle \Phi,  \Psi \rangle_{L^2(X, dv_{\tilde{g}})}|^2 = \sum_{\Psi} \left| \int_{X}  \overline{ \Phi}\Psi \, dv_{\tilde{g}} \right|^2.
\end{equation}
Denoting by $\frac{dv_{\tilde{g}}}{dv_g}$ the Radon-Nikodoym density of the measure $v_{\tilde{g}}$ with respect to $v_g$, we have
\begin{equation}
\sum_{\Psi} \left| \int_{X}  \overline{\Phi}\Psi \, dv_{\tilde{g}} \right|^2=
\sum_{\Psi} \left| \int_{X}  \overline{\frac{dv_{\tilde{g}}}{dv_g} \Phi}\Psi \, dv_{{g}} \right|^2=\sum_{\Psi} \left|\left\langle \frac{dv_{\tilde{g}}}{dv_g}\Phi,  \Psi \right\rangle_{L^2(X, dv_{g})}\right|^2.
\end{equation}
Next, we use that sum of the projections onto the $\Psi$ is an operator on $L^2(X, dv_{g})$ that is bounded by $1$, so 
\begin{equation}
\begin{split}
\sum_{\Psi} \left|\left\langle \frac{dv_{\tilde{g}}}{dv_g}\Phi,  \Psi \right\rangle_{L^2(X, dv_{g})}\right|^2\le \left\| \left( \frac{dv_{\tilde{g}}}{dv_g} \right)   \Phi\right\|^2_{L^2(X, dv_{g})}
=\int_{X} \left| \left( \frac{dv_{\tilde{g}}}{dv_g} \right)  \Phi\right|^2 \, dv_g\\
\le(1+\delta)\int_{X} \left|    \Phi\right|^2 \, dv_{\tilde{g}}\le(1+\delta)\int_{X} \left|   \Phi\right|^2 \, dv_{\tilde{g}}.
\end{split}
\end{equation}
Here we used that by the quasi-isometry, see \eqref{eq:volume_quasi_isometry}, $\frac{dv_{\tilde{g}}}{dv_g}\le 1+\delta$. This proves \eqref{eq:projection_estimate}.

Taking the sum over all $\Psi$ in \eqref{eq:chi_Psi_ohne_Tr} yields
\begin{align*}
\begin{split}
\sum_{\Psi} \int_{X } \chi^2|\Psi|^2 \, dv_g &\leq \frac{1}{\lambda} (1 + \delta)^2 \Tr_{L^2(X , dv_{\tilde{g}})} \left(\chi \left(2 \lambda - \frac{1}{(1+\delta)^2}\tilde\Delta\right) \chi\right)_+ + \frac{1}{\lambda} \sum_{\Psi} \int_{X} |\nabla\chi|^2 |\Psi|^2 \, dv_g\\
&=\frac{1}{\lambda}  \Tr_{L^2(X, dv_{\tilde{g}})} \left( \chi\left(2(1 + \delta)^2 \lambda - \tilde\Delta\right)\chi \right)_+ + \frac{1}{\lambda} \sum_{\Psi} \int_{X} |\nabla\chi|^2 |\Psi|^2 \, dv_g.
\end{split}
\end{align*}

{\bf End of the proof of $(ii)$. }
Using the bra-ket notation, we claim that
\begin{equation}\label{eq:projection_estimate_chi_psi}
0 \leq \sum_{\Psi} |\chi \Psi \rangle \langle \chi \Psi| \leq (1 + \delta) 1_{X \setminus X_a} \quad \text{on } L^2(X, dv_{\tilde{g}}),
\end{equation}
that is, the sum of the orthogonal projections onto $\chi \Psi$ in $L^2(X, dv_{\tilde{g}})$ is less than or equal in an operator sense to $1+\delta$ times the multiplication operator by $1_{X \setminus X_a}$. Note that here, the orthogonal projection onto $\chi \Psi$ in $L^2(X, dv_{\tilde{g}})$ is different from the orthogonal projection onto $\chi \Psi$ in $L^2(X, dv_{g})$. 
To see \eqref{eq:projection_estimate_chi_psi}, let $\Phi \in L^2(X \setminus X_a, dv_{\tilde{g}})$. Then
\begin{equation}
\left\langle \Phi, \sum_{\Psi} |\chi \Psi \rangle \langle \chi \Psi | \Phi  \right\rangle_{L^2(X, dv_{\tilde{g}})} = \sum_{\Psi} |\langle \chi\Phi,  \Psi \rangle_{L^2(X, dv_{\tilde{g}})}|^2 = \sum_{\Psi} \left| \int_{X}  \overline{\chi \Phi}\Psi \, dv_{\tilde{g}} \right|^2.
\end{equation}
Denoting by $\frac{dv_{\tilde{g}}}{dv_g}$ the Radon-Nikodoym density of the measure $v_{\tilde{g}}$ with respect to $v_g$, we have
\begin{equation}
\sum_{\Psi} \left| \int_{X}  \overline{\chi \Phi}\Psi \, dv_{\tilde{g}} \right|^2=
\sum_{\Psi} \left| \int_{X}  \overline{\frac{dv_{\tilde{g}}}{dv_g}\chi \Phi}\Psi \, dv_{{g}} \right|^2=\sum_{\Psi} \left|\left\langle \frac{dv_{\tilde{g}}}{dv_g}\chi\Phi,  \Psi \right\rangle_{L^2(X, dv_{g})}\right|^2.
\end{equation}
Next, we use that sum of the projections onto the $\Psi$ is an operator on $L^2(X, dv_{g})$ that is bounded by $1$, so 
\begin{equation}
\begin{split}
\sum_{\Psi} \left|\left\langle \frac{dv_{\tilde{g}}}{dv_g}\chi\Phi,  \Psi \right\rangle_{L^2(X, dv_{g})}\right|^2\le \left\| \left( \frac{dv_{\tilde{g}}}{dv_g} \right)  \chi \Phi\right\|^2_{L^2(X, dv_{g})}
=\int_{X} \left| \left( \frac{dv_{\tilde{g}}}{dv_g} \right)  \chi \Phi\right|^2 \, dv_g\\
\le(1+\delta)\int_{X} \left|   \chi \Phi\right|^2 \, dv_{\tilde{g}}\le(1+\delta)\int_{X\setminus X_a} \left|   \Phi\right|^2 \, dv_{\tilde{g}}.
\end{split}
\end{equation}
Here we used that by the quasi-isometry, see \eqref{eq:volume_quasi_isometry}, $\frac{dv_{\tilde{g}}}{dv_g}\le 1+\delta$. This proves \eqref{eq:projection_estimate_chi_psi}.

Taking the sum over all $\Psi$ in \eqref{eq:chi_Psi_ohne_Tr} yields
\begin{align*}
\begin{split}
\sum_{\Psi} \int_{X \setminus X_b} |\Psi|^2 \, dv_g &\leq \frac{1}{\lambda} (1 + \delta)^2 \Tr_{L^2(X \setminus X_a, dv_{\tilde{g}})} \left( \left(2 \lambda - \frac{1}{(1+\delta)^2}\tilde\Delta\right)_+ \right) + \frac{1}{\lambda} \sum_{\Psi} \int_{X} |\chi'|^2 |\Psi|^2 \, dv_g\\
&=\frac{1}{\lambda}  \Tr_{L^2(X \setminus X_a, dv_{\tilde{g}})} \left( \left(2(1 + \delta)^2 \lambda - \tilde\Delta\right)_+ \right) + \frac{1}{\lambda} \sum_{\Psi} \int_{X} |\chi'|^2 |\Psi|^2 \, dv_g.
\end{split}
\end{align*}
\end{proof}

\section{Estimates for the number of eigenvalues of $\tilde\Delta$}\label{s:number_eigenvalues}

In this section, we prove optimal estimates for the number of eigenvalues of the separable operator $\tilde\Delta$ in different subsets near the boundary. 

\begin{proposition}\label{prop:number_eigenvalues}
Let $\epsilon\in(0,1]$ and let $\tilde\Delta$ be the Laplace-Beltrami operator corresponding to $\tilde g$, see \eqref{eq:tilde_Delta}. Let $\beta\in[2/n,\infty)$. 
There exists a constant $C>0$ independent of $\lambda$ such that for all $B_{\lambda}$ with 
\begin{equation}
\frac{1}{2}\lambda^{-1/2}\le B_{\lambda}\le\frac{\epsilon}{2}
\end{equation}
and 
for all $\lambda\ge 1$, we have the following estimates: 
\begin{itemize}
\item[(i)] If $\beta>2/n$, then for the operator $\tilde\Delta$ on $L^2(X \setminus X_{B_\lambda}, dv_{\tilde{g}})$ with Neumann boundary conditions at $\partial\left(X \setminus X_{B_\lambda}\right)$, the number of its eigenvalues below $\lambda$, denoted by $N_{\tilde\Delta,X \setminus X_{B_\lambda}}(\lambda)$, satisfies 
\begin{equation}\label{eq:eigenvalue_est_super}
N_{\tilde\Delta,X \setminus X_{B_\lambda}}(\lambda)\le C\lambda^{\frac{n+1}{2}}B_{\lambda}^{-\frac{\beta n}{2}\left(1-\frac{2}{\beta n}\right)}.
\end{equation}
\item[(ii)] If $\beta=2/n$, then for the operator $\tilde\Delta$ on $L^2(X \setminus X_{B_\lambda}, dv_{\tilde{g}})$ with Neumann boundary conditions at $\partial\left(X \setminus X_{B_\lambda}\right)$, we have
\begin{equation}\label{eq:eigenvalue_est_crit}
N_{\tilde\Delta,X \setminus X_{B_\lambda}}(\lambda)\le C\lambda^{\frac{n+1}{2}}\left(1+|\log(B_{\lambda})|\right).
\end{equation}
\item[(iii)] If $\beta\ge2/n$, then for the operator $\tilde\Delta$ on $L^2(X_{2B_\lambda} \setminus X_{B_\lambda}, dv_{\tilde{g}})$ with Neumann boundary conditions at $\partial\left(X_{2B_\lambda} \setminus X_{B_\lambda}\right)$, we have 
\begin{equation}
N_{\tilde\Delta,X_{2B_\lambda} \setminus X_{B_\lambda}}(\lambda)\le C\lambda^{\frac{n+1}{2}}B_{\lambda}^{-\frac{\beta n}{2}\left(1-\frac{2}{\beta n}\right)}.
\end{equation}
\item[(iv)] If $\beta\ge2/n$, then for the operator $\tilde\Delta$ on $L^2(X_{\frac{1}{2}\lambda^{-1/2}}, dv_{\tilde{g}})$ with Neumann boundary conditions at $\partial X_{\frac{1}{2}\lambda^{-1/2}}$, we have 
\begin{equation}
N_{\tilde\Delta,X_{\frac{1}{2}\lambda^{-1/2}} }(3\lambda)=0.
\end{equation}
\end{itemize}
Furthermore, the result is optimal in the sense that if in addition, for some constant $\hat c>0$ large enough depending on $n$ and $\beta$, and for some $\tilde\epsilon>0$ small enough, we have
\begin{equation}
\frac{1}{2}\lambda^{-1/2}\le B_{\lambda}\le\frac{\epsilon}{2}, 
\end{equation}
then for all $\lambda>0$ large enough, we have
\item[(v)] if $\beta>2/n$, then
\begin{equation}\label{eq:eigenvalue_est_super_opt}
N^D_{\tilde\Delta,X \setminus X_{B_\lambda}}(\lambda)\ge C\lambda^{\frac{n+1}{2}}B_{\lambda}^{-\frac{\beta n}{2}\left(1-\frac{2}{\beta n}\right)}.
\end{equation}
Here the $D$ stands for Dirichlet boundary conditions on $\partial\left(X \setminus X_{B_\lambda}\right)$. 
\item[(vi)] if $\beta=2/n$, then
\begin{equation}\label{eq:eigenvalue_est_crit_opt}
N^D_{\tilde\Delta,X \setminus X_{B_\lambda}}(\lambda)\ge C\lambda^{\frac{n+1}{2}}\left(1+|\log(B_{\lambda})|\right).
\end{equation}
Here the $D$ stands for Dirichlet boundary conditions on $\partial\left(X \setminus X_{B_\lambda}\right)$. 
\item[(vii)] If $\beta\ge2/n$, then for the operator $\tilde\Delta$ on $L^2(X_{2B_\lambda} \setminus X_{B_\lambda}, dv_{\tilde{g}})$ with Neumann boundary conditions at $\partial\left(X_{2B_\lambda} \setminus X_{B_\lambda}\right)$, we have 
\begin{equation}
N^D_{\tilde\Delta,X_{2B_\lambda} \setminus X_{B_\lambda}}(\lambda)\ge C\lambda^{\frac{n+1}{2}}B_{\lambda}^{-\frac{\beta n}{2}\left(1-\frac{2}{\beta n}\right)}.
\end{equation}
Here the $D$ stands for Dirichlet boundary conditions on $\partial\left(X_{2B_\lambda} \setminus X_{B_\lambda}\right)$. 
\end{proposition}
\begin{proof}
First note that we have
\begin{equation}
N_{\tilde\Delta,X \setminus X_{B_\lambda}}(\lambda)\le N_{\tilde\Delta,X_\epsilon \setminus X_{B_\lambda}}(\lambda) +N_{\tilde\Delta,X \setminus X_{\epsilon}}(\lambda),
\end{equation} 
where we put Neumann boundary conditions on all the boundaries. Since the metric $\tilde g $ is regular in the interior of $X$, we have for all $\lambda\ge 1$
\begin{equation}
N_{\tilde\Delta,X \setminus X_{\epsilon}}(\lambda)\le C\lambda^{\frac{n+1}{2}}
\end{equation}
for a constant $C$ that only depends on $\epsilon$, $n$ and $\beta$. Note that this is a subleading term for both \eqref{eq:eigenvalue_est_super} and \eqref{eq:eigenvalue_est_crit}.

\medskip

Next, let us estimate $N_{\tilde\Delta,X_\epsilon \setminus X_{B_\lambda}}(\lambda)$. On $X_\epsilon$, and hence also on $X_\epsilon \setminus X_{B_\lambda}$, the metric $\tilde g$ reads 
\begin{equation}
\tilde g = \mathrm{d}x^2+x^{-\beta} G
\end{equation}
and after a change of function and measure, see \cite[Section 3.8]{PhD}, the corresponding Laplace-Beltrami operator $\tilde\Delta$ can be written as
\begin{equation}
\tilde\Delta = -\partial_x^2+\frac{C_\beta}{x^2}+x^\beta\Delta_M, \quad \text{on } L^2\left([0,\epsilon]\times M, dx\, dv_G(y)\right),\quad C_\beta:=\frac{\beta n}{4}\left(1+\frac{\beta n}{4}\right),
\end{equation}
where $\Delta_M$ denotes the Laplace-Beltrami operator on $M$ with respect to $G$. By \eqref{eq:tilde_Delta}, a basis of eigenfunctions of $\tilde\Delta$ is such that the eigenfunctions factorise. More precisely, if $\Phi$ is an eigenfunction in this basis of $\tilde\Delta$ with eigenvalue $\alpha$, then it can be written as 
\begin{equation}
\Phi(x,y)=\phi(x)\psi(y),
\end{equation}
where $\psi$ is an eigenfunction with some eigenvalue $\mu$ of $\Delta_M$, and $\phi$ is an eigenfunction of the operator
\begin{equation}
P_\mu :=-\partial_x^2+\frac{C_\beta}{x^2}+\mu x^\beta \quad \text{on } [B_\lambda, \epsilon].
\end{equation}
We denote the eigenvalues of $\Delta_M$ by $(\mu_j)_{j\in\N}$. By the Weyl law on $M$, we know that $\mu_j\sim c_M j^{2/n}$ as $n\to\infty$, where $c_M$ is an explicit constant that only depends on $v_G(M)$ and $n$. In the following, to simplify the notation, we will assume that $c_M=1$. Let us denote by $N_{\mu,[a,b]}(\lambda)$ the number of eigenvalues $<\lambda$ of $P_\mu$ on the interval $[a,b]$ with Neumann boundary conditions at both endpoints. Then we get
\begin{equation}
N_{\tilde\Delta,X_\epsilon \setminus X_{B_\lambda}}(\lambda)=\sum_{j\in\N}N_{\mu_j,[B_\lambda, \epsilon]}(\lambda).
\end{equation}
Also note that for any $\mu\ge 0$, we have $P_\mu \ge-\partial_x^2$, and by estimating the number of eigenvalues of $-\partial_x^2$ on intervals, we get for any $j\in\N$
\begin{equation}
N_{\mu_j,[B_\lambda, \epsilon]}(\lambda)\le 1+\frac{\epsilon}{\pi}\sqrt{\lambda}.
\end{equation}
In particular, 
\begin{equation}
\sum_{j=1}^{\lambda^{n/2}}N_{\mu_j,[B_\lambda, \epsilon]}(\lambda)\le C\lambda^{(n+1)/2}, 
\end{equation}
which is less than or equal to the right-hand side of the desired estimate. 

For the part of the sum corresponding to $j\ge\lambda^{n/2}$, we may replace $\mu_j$ by $j^{2/n}$ up to a constant factor in the estimates, which will not matter. So it remains to estimate
\begin{equation}
\sum_{j=\lambda^{n/2}}^{\infty}N_{j^{2/n},[B_\lambda, \epsilon]}(\lambda).
\end{equation} 
Note that if $j\ge\lambda^{n/2}B_\lambda^{-\beta n/2}$, then $j^{2/n}B_\lambda\ge\lambda$, and thus, $P_{j^{2/n}}\ge\lambda$ on $[B_\lambda, \epsilon]$. That is, for $j\ge\lambda^{n/2}B_\lambda^{-\beta n/2}$, we have $N_{j^{2/n},[B_\lambda, \epsilon]}(\lambda)=0$. Hence, we only need to estimate 
\begin{equation}
\sum_{j=\lambda^{n/2}}^{\lambda^{n/2}B_\lambda^{-\beta n/2}}N_{j^{2/n},[B_\lambda, \epsilon]}(\lambda).
\end{equation} 
Using the unitary equivalence of $P_\mu$ on an interval $[a,b]$ and $\mu^{2/(2+\beta)}P_1$ on $\left[\mu^{1/(2+\beta)}a,\mu^{1/(2+\beta)}b\right]$, we have
\begin{equation}
\sum_{j=\lambda^{n/2}}^{\lambda^{n/2}B_\lambda^{-\beta n/2}}N_{j^{2/n},[B_\lambda, \epsilon]}(\lambda)
=\sum_{j=\lambda^{n/2}}^{\lambda^{n/2}B_\lambda^{-\beta n/2}}N_{1,[j^{1/d}B_\lambda, j^{1/d}\epsilon]}(\lambda j^{-2/d}),
\end{equation} 
where
\begin{equation}
d:=n\left(1+\frac{\beta}{2}\right).
\end{equation}
Since $B_\lambda\ge\frac{1}{2}\lambda^{-1/2}$, we have 
$\lambda^{n/2}B_\lambda^{-\beta n/2}\le C_{n,\beta} \lambda^{d/2}$ and thus, $\lambda j^{-2/d}\ge C_{n,\beta}$, where $C_{n,\beta}>0$ is a constant only depending on $n$, $\beta$, and which may change in every equation. By \cite[(3.43)]{PhD}, there is a constant $C>0$ independent of $\lambda$ such that  
\begin{equation}
\sum_{j=\lambda^{n/2}}^{\lambda^{n/2}B_\lambda^{-\beta n/2}}N_{1,[j^{1/d}B_\lambda, j^{1/d}\epsilon]}(\lambda j^{-2/d})
\le C\sum_{j=\lambda^{n/2}}^{\lambda^{n/2}B_\lambda^{-\beta n/2}}\left(\lambda j^{-\frac{2}{d}}\right)^{\frac{1}{2}+\frac{1}{\beta}}
= C\lambda^{\frac{1}{2}+\frac{1}{\beta}}\sum_{j=\lambda^{n/2}}^{\lambda^{n/2}B_\lambda^{-\beta n/2}}j^{-\frac{2}{n\beta}}.
\end{equation} 
If $\beta>2/n$, then we can estimate for a constant $C>0$ depending on $n$ and $\beta$
\begin{equation}
\lambda^{\frac{1}{2}+\frac{1}{\beta}}\sum_{j=\lambda^{n/2}}^{\lambda^{n/2}B_\lambda^{-\beta n/2}}j^{-\frac{2}{n\beta}}\le C \lambda^{\frac{1}{2}+\frac{1}{\beta}}\left(\lambda^{n/2}B_\lambda^{-\beta n/2}\right)^{1-\frac{2}{n\beta}}
=C\lambda^{\frac{n+1}{2}}B_\lambda^{-\frac{\beta n}{2}\left(1-\frac{2}{n\beta}\right)},
\end{equation}
which proves $(i)$. 

\medskip

If $\beta=2/n$, then 
\begin{equation}
\lambda^{\frac{1}{2}+\frac{1}{\beta}}\sum_{j=\lambda^{n/2}}^{\lambda^{n/2}B_\lambda^{-\beta n/2}}j^{-\frac{2}{n\beta}}
=\lambda^{\frac{n+1}{2}}\sum_{j=\lambda^{n/2}}^{\lambda^{n/2}B_\lambda^{-1}}j^{-1}
\le C\lambda^{\frac{n+1}{2}}\left(1+\log\left(\frac{\lambda^{n/2}B_\lambda^{-1}}{\lambda^{n/2}}\right)\right)
=C\lambda^{\frac{n+1}{2}}\left(1+\left|\log\left(B_\lambda\right)\right|\right).
\end{equation}
This shows $(ii)$. 

\medskip

For the proof of $(iii)$, note that similar to above, we need to estimate
\begin{equation}\label{eq:B_lambda_2_start}
\sum_{j=\lambda^{n/2}}^{\lambda^{n/2}B_\lambda^{-\beta n/2}}N_{j^{2/n},[B_\lambda, 2B_\lambda]}(\lambda)
\end{equation} 
Now, since we have for any $\mu>0$ that $P_\mu\ge-\partial_x^2$, and since $B_\lambda\ge\frac{1}{2}\lambda^{-1/2}$, we find
\begin{equation}
N_{j^{2/n},[B_\lambda, 2B_\lambda]}(\lambda)\le C\sqrt{\lambda}B_\lambda.
\end{equation}
Combining this with \eqref{eq:B_lambda_2_start}, we obtain
\begin{equation}
\sum_{j=\lambda^{n/2}}^{\lambda^{n/2}B_\lambda^{-\beta n/2}}N_{j^{2/n},[B_\lambda, 2B_\lambda]}(\lambda)\le C\lambda^{(n+1)/2}B_\lambda^{1-\beta n/2}=C\lambda^{(n+1)/2}B_\lambda^{\frac{n\beta}{2}\left(1-\frac{2}{n\beta}\right)}.
\end{equation}

\medskip

For $(iv)$, we need to show that 
\begin{equation}
N_{\tilde\Delta, X_{\frac{1}{2}\lambda^{-1/2}}}(3\lambda)=\sum_{j\in\N}N_{\mu_j,[0,\frac{1}{2}\lambda^{-1/2}]}(3\lambda).
\end{equation}
is zero. To see this, note that for any $\mu>0$, we have on $[0,\frac{1}{2}\lambda^{-1/2}]$
\begin{equation}
P_\mu\ge \frac{C_\beta}{x^2}\ge4\lambda C_\beta\ge 3\lambda,
\end{equation}
where we have used that $\beta\ge 2/n$ and 
\begin{equation}
C_\beta:=\frac{\beta n}{4}\left(1+\frac{\beta n}{4}\right)\ge\frac{3}{4}.
\end{equation}

\bigskip

Let us now turn to the proof of the optimality statements $(v)$, $(vi)$ and $(vii)$. First note that for any $\hat c>0$, if $B_\lambda\ge \hat c\lambda^{-\frac{1}{2}}$, then we have for all $x\ge B_\lambda$
\begin{equation}
\frac{C_\beta}{x^2}\le\frac{C_\beta}{\hat c^2}\lambda. 
\end{equation}
Thus, if we choose $\hat c$ depending on $n,\beta$ large enough, then 
\begin{equation}\label{eq:C_beta_x_2_est}
\frac{C_\beta}{x^2}\le\frac{\lambda}{2}. 
\end{equation}
We define the operator 
\begin{equation}
\hat P_\mu:=-\partial_x^2+\mu x^\beta
\end{equation}
and note that by \eqref{eq:C_beta_x_2_est} 
both for Dirichlet boundary conditions at both endpoints of the interval, denoted by $D$, we have on any interval $I$ of the form $[B_\lambda,\epsilon]$ or $[B_\lambda,2B_\lambda]$ 
\begin{equation}\label{eq:P_hat_P_halbe_est}
N^{D}_{\mu,I}(\lambda)\ge \hat N^{D}_{\mu,I}\left(\frac{\lambda}{2}\right)
\end{equation}
Let us note that the boundary conditions are not important here, but we will only need Dirichlet boundary conditions in the following. 
Here $\hat N$ denotes the corresponding number of eigenvalues of $\hat P_{\mu}$ on the interval $I$. 

\medskip

To see the optimality, note that for $(v)$ and $(vi)$, it suffices to prove the corresponding bound for 
\begin{equation}
\sum_{j=\epsilon^{-\beta n/2}\lambda^{n/2}}^{\lambda^{n/2}B_\lambda^{-\beta n/2}}N^D_{j^{2/n},[B_\lambda, \epsilon]}(\lambda).
\end{equation}
Here the $D$ stands for Dirichlet boundary conditions at both endpoints. The fact that we replaced $\lambda^{n/2}$ in the lower bound of the sum by $\epsilon^{-\beta n/2}\lambda^{n/2}$ does not change anything for the proof, but it simplifies our notation in the following. First note that by \eqref{eq:P_hat_P_halbe_est}, we have
\begin{equation}\label{eq:opt_gross_start}
\sum_{j=\epsilon^{-\beta n/2}\lambda^{n/2}}^{\lambda^{n/2}B_\lambda^{-\beta n/2}}N^D_{j^{2/n},[B_\lambda, \epsilon]}(\lambda)
\ge\sum_{j=\epsilon^{-\beta n/2}\lambda^{n/2}}^{\lambda^{n/2}B_\lambda^{-\beta n/2}}\hat N^D_{j^{2/n},[B_\lambda, \epsilon]}\left(\frac \lambda 2\right)
= \sum_{j=\epsilon^{-\beta n/2}\lambda^{n/2}}^{\lambda^{n/2}B_\lambda^{-\beta n/2}}\hat N^D_{1,[j^{1/d}B_\lambda, j^{1/d}\epsilon]}\left(\frac{\lambda j^{-2/d}}{2}\right),
\end{equation}
where we used the unitary equivalence of $\hat P_\mu$ on an interval $[a,b]$ and $\mu^{2/(2+\beta)}\hat P_1$ on $\left[\mu^{1/(2+\beta)}a,\mu^{1/(2+\beta)}b\right]$ in the last step. 

We use the notation $\omega:=\frac{1}{2}\lambda j^{-2/d}$. Note that 
\begin{equation}
\text{if } j\ge\left(\frac{1}{\epsilon}\right)^{\frac{n\beta}{2}}\left(\frac{1}{2}\right)^{\frac{n}{2}}\lambda^{\frac{n}{2}}, \, 
\text{ then } j^{\frac{1}{d}} \epsilon\ge\omega^{\frac{1}{\beta}}
\end{equation}
and 
\begin{equation}
\text{if } j\le\left(\frac{1}{2}\right)^{\frac{n\beta}{2}+\frac{n\beta}{2}}B_\lambda^{-\frac{n\beta}{2}}\lambda^{\frac{n}{2}}, \, 
\text{ then }  j^{\frac{1}{d}} B_\lambda\le\frac{1}{2}\omega^{\frac{1}{\beta}}.
\end{equation}
It follows that, using the notation $\omega(j,\lambda):=\frac{1}{2}\lambda j^{-2/d}$, 
\begin{equation}\label{eq:opt_gross_omega}
\sum_{j=\epsilon^{-\beta n/2}\lambda^{n/2}}^{\lambda^{n/2}B_\lambda^{-\beta n/2}}\hat N^D_{1,[j^{1/d}B_\lambda, j^{1/d}\epsilon]}\left(\frac{\lambda j^{-2/d}}{2}\right)
\ge \sum_{j=\epsilon^{-\beta n/2}\lambda^{n/2}}^{2^{-\frac{n}{2}-\frac{n\beta}{2}}\lambda^{n/2}B_\lambda^{-\beta n/2}}\hat N^D_{1,[\frac{1}{2}\omega(j,\lambda)^{1/\beta}, \omega(j,\lambda)^{1/\beta}]}\left(\omega(j,\lambda)^{1/\beta}\right)
\end{equation}

Similar to our remark earlier in the proof, since $B_\lambda\ge\hat c\lambda^{-1/2}$, we have $\omega\ge C_{n,\beta}>0$. Similar to \cite[(3.31)]{PhD}, there exists a constant $c>0$ such that for all $\omega\ge C_{n,\beta}$, we have
\begin{equation}\label{eq:opt_P_Hut_untere_Schranke}
\hat N^D_{1,\left[\frac{1}{2} \omega^{1/\beta}, \omega^{1/\beta}\right]}\left(\omega\right)\ge c\omega^{\frac{1}{2}+\frac{1}{\beta}}. 
\end{equation}
Thus, we obtain from \eqref{eq:opt_gross_start}, \eqref{eq:opt_gross_omega} and \eqref{eq:opt_P_Hut_untere_Schranke} 
\begin{equation}
\sum_{j=\epsilon^{-\beta n/2}\lambda^{n/2}}^{\lambda^{n/2}B_\lambda^{-\beta n/2}}N^D_{j^{2/n},[B_\lambda, \epsilon]}(\lambda)
\ge C \sum_{j=\epsilon^{-\beta n/2}\lambda^{n/2}}^{2^{-\frac{n}{2}-\frac{n\beta}{2}}\lambda^{n/2}B_\lambda^{-\beta n/2}} \left(\lambda j^{-2/d}\right)^{\frac{1}{2}+\frac{1}{\beta}}.
\end{equation}
For the rest of the proofs of $(v)$ and $(vi)$, one can follow the calculations at the end of the proofs of $(i)$ and $(ii)$. 

\bigskip

Let us now turn to the proof of $(vii)$. By \eqref{eq:P_hat_P_halbe_est}, it suffices to prove the corresponding bound for
\begin{equation}
\sum_{j=\lambda^{n/2}}^{\lambda^{n/2}B_\lambda^{-\beta n/2}}\hat N^D_{j^{2/n},[B_\lambda,2B_\lambda]}\left(\frac \lambda 2\right)
\end{equation}
Note that
\begin{equation}
\text{if } j\le 2^{-n-\frac{n\beta}{2}}\lambda^{\frac{n}{2}}B_\lambda^{-\frac{\beta n}{2}},\, \text{ then } j^{\frac{2}{n}}\left(2 B_\lambda\right)^\beta\le\frac{\lambda}{4}.
\end{equation}
Thus, we have 
\begin{align}
\begin{split}
&\quad \sum_{j=\lambda^{n/2}}^{\lambda^{n/2}B_\lambda^{-\beta n/2}}\hat N^D_{j^{2/n},[B_\lambda,2B_\lambda]}\left(\frac \lambda 2\right)\\
&\ge
\sum_{j=2^{-1-n-n\beta/2}\lambda^{n/2}B_\lambda^{-\beta n/2}}^{2^{-n-n\beta/2}\lambda^{n/2}B_\lambda^{-\beta n/2}} 
\#\left\{
\text{number of eigenvalues of } -\partial_x^2 \text{ on } 
[0,B_\lambda] \text { that are } < \frac{\lambda}{4}
\right\}\\
&\ge C\lambda^{n/2}B_\lambda^{-\beta n/2} \sqrt{\lambda}B_\lambda
=C\lambda^{\frac{n+1}{2}}B_\lambda^{1-\frac{n\beta}{2}},
\end{split}
\end{align}
which is the desired result.
\end{proof}

\section{Estimate of the localisation error and conclusion}\label{s:conclusion}
In this section, we give the proof Theorem \ref{th:main} combining the estimates from Proposition \ref{prop:localisation}(ii) and Proposition \ref{prop:number_eigenvalues}(i),(ii). In the critical case $\beta=2/n$, this is straightforwardly possible. In the supercritical case $\beta>2/n$, we use an additional bootstrap argument to improve the final estimate in Theorem \ref{th:main}. 

\subsection{Proof of Theorem \ref{th:main} for $\beta=2/n$}

\begin{proof}[Proof of Theorem \ref{th:main} for $\beta=2/n$]
Let $\beta=2/n$. Choose $\epsilon>0$ small enough such that the corresponding $\delta$ in \eqref{eq:g_g_tilde_est} satisfies $\delta\le1$. Recall from Remark \ref{re:xF_von_Indikator_beim_Rand} that it suffices to prove
\begin{equation}\label{eq:goal_th_proof_recall_crit}
\int_{X\setminus X_{A_\lambda}} F_\lambda\,dv_g\le C A_\lambda. 
\end{equation}
We first apply Proposition \ref{prop:localisation}(ii) with $a:=A_\lambda/2$ and $b=A_\lambda$. 
In particular, we can take a smooth function $\chi : [0, \epsilon] \to [0, 1]$ with $\chi(x) = 0$ for $x \leq A_\lambda/2$ and $\chi(x) = 1$ for $x \geq A_\lambda$, and such that 
\begin{equation}\label{eq:chi_der}
|\chi'(x)|\le\frac{4}{A_\lambda}\quad\text{for all } x\in [0, \epsilon].
\end{equation}
By Proposition \ref{prop:localisation}(ii), and \eqref{eq:chi_der}, we have
\begin{align}
\begin{split}
&\quad
 \sum_{\substack{\Psi \text{ eigenfunction} \\ \text{of } \Delta \text{ with eigenvalue } < \lambda}} \int_{X \setminus X_{A_\lambda}} |\Psi|^2 \, dv_g \\
&\leq \frac{1}{\lambda} \Tr_{L^{2}(X \setminus X_{\frac{A_\lambda}{2}}, dv_{\tilde g})} \left( (2(1+\delta)^2 \lambda - \tilde{\Delta})_{+} \right) + \frac{1}{\lambda} \sum_{\substack{\Psi \text{ eigenfunction} \\ \text{of } \Delta \text{ with eigenvalue } < \lambda}} \int_{X} |\chi'|^2 |\Psi|^2 \, dv_g\\
&\leq \frac{1}{\lambda} \Tr_{L^{2}(X \setminus X_{\frac{A_\lambda}{2}}, dv_{\tilde g})} \left( (2(1+\delta)^2 \lambda - \tilde{\Delta})_{+} \right) + \frac{16}{\lambda A_\lambda^2} \sum_{\substack{\Psi \text{ eigenfunction} \\ \text{of } \Delta \text{ with eigenvalue } < \lambda}} \int_{X} |\Psi|^2 \, dv_g.
\end{split}
\end{align}
Now, dividing by $N(\lambda)$ and using the definition of $F_\lambda$, see \eqref{eq:F_lambda_def}, we obtain 
\begin{align}\label{eq:F_lambda_first_est_crit_proof}
\begin{split}
&\quad\quad\int_{X\setminus X_{A_\lambda}} F_\lambda\,dv_g= 
\frac{1}{N(\lambda)}\sum_{\substack{\Psi \text{ eigenfunction} \\ \text{of } \Delta \text{ with eigenvalue } < \lambda}} \int_{X \setminus X_{A_\lambda}} |\Psi|^2 \, dv_g \\
&\leq \frac{1}{\lambda N(\lambda)} \Tr_{L^{2}(X \setminus X_{\frac{A_\lambda}{2}}, dv_{\tilde g})} \left( (2(1+\delta)^2 \lambda - \tilde{\Delta})_{+} \right) + \frac{16}{\lambda A_\lambda^2}.
\end{split}
\end{align}
Note that by Proposition \ref{prop:number_eigenvalues}(ii) applied with $B_\lambda=A_\lambda/2$ and $\lambda$ replaced by $8\lambda$, we have
\begin{align}
\begin{split}
&\Tr_{L^{2}(X \setminus X_{\frac{A_\lambda}{2}}, dv_{\tilde g})} \left( (2(1+\delta)^2 \lambda - \tilde{\Delta})_{+} \right)
\le 2(1+\delta)^2 \lambda N_{\tilde\Delta,X \setminus X_{\frac{A_\lambda}{2}}}(2(1+\delta)^2 \lambda)\\
&\le 8 \lambda N_{\tilde\Delta,X \setminus X_{\frac{A_\lambda}{2}}}(8 \lambda)
\le \lambda C \lambda^{\frac{n+1}{2}}|\log(8A_{\lambda})|\le C \lambda^{1+\frac{n+1}{2}}|\log(A_{\lambda})|.
\end{split}
\end{align}
Plugging this into \eqref{eq:F_lambda_first_est_crit_proof} and using that for $\beta=2/n$,
\begin{equation}
N(\lambda)\sim C\lambda^{(n+1)/2}\log(\lambda), \quad \text{as } \lambda\to\infty,
\end{equation}
we get 
\begin{equation}\label{eq:F_lambda_proof_crit_est}
\int_{X\setminus X_{A_\lambda}} F_\lambda\,dv_g\le C\frac{|\log(A_{\lambda})|}{\log(\lambda)}+\frac{16}{\lambda A_\lambda^2}.
\end{equation}
In Theorem \ref{th:main}(ii), namely if $\beta=2/n$, we choose 
\begin{equation}
A_\lambda:=\frac{\log(\log(\lambda))}{\log(\lambda)}.
\end{equation}
Using this choice of $A_\lambda$ in \eqref{eq:F_lambda_proof_crit_est} and $1/\lambda=o(A_\lambda^3)$ as $\lambda\to\infty$, we obtain
\begin{align}
\begin{split}
\int_{X\setminus X_{A_\lambda}} F_\lambda\,dv_g\le C\frac{\log(A_{\lambda}^{-1})}{\log(\lambda)}+\frac{16}{\lambda A_\lambda^2}
\le C\frac{\log\left(\frac{\log(\lambda)}{\log(\log(\lambda))}\right)}{\log(\lambda)}+CA_\lambda\le 2C A_\lambda. 
\end{split}
\end{align}
This finishes the proof of \eqref{eq:goal_th_proof_recall_crit} for $\beta=2/n$. 
\end{proof}

\subsection{Proof of Theorem \ref{th:main} for $\beta>2/n$}
\begin{proof}[Proof of Theorem \ref{th:main} for $\beta>2/n$]
Let $\beta>2/n$. Choose $\epsilon>0$ small enough such that the corresponding $\delta$ in \eqref{eq:g_g_tilde_est} satisfies $\delta\le1$.  Recall from Remark \ref{re:xF_von_Indikator_beim_Rand} that it suffices to prove
\begin{equation}\label{eq:goal_th_proof_recall_super}
\int_{X\setminus X_{A_\lambda}} F_\lambda\,dv_g\le C A_\lambda. 
\end{equation}
For $B_\lambda$ to be chosen, and $k\in\N$ to be chosen, we apply Proposition \ref{prop:localisation}(ii) with $a:=kB_\lambda$ and $b=(k+1)B_\lambda$. In particular, we can take a smooth function $\chi : [0, \epsilon] \to [0, 1]$ with $\chi(x) = 0$ for $x \leq kB_\lambda$ and $\chi(x) = 1$ for $x \geq kB_\lambda$, and such that 
\begin{equation}\label{eq:chi_der2}
|\chi'(x)|\le\frac{2}{B_\lambda}\quad\text{for all } x\in [0, \epsilon].
\end{equation}
By Proposition \ref{prop:localisation}(ii), and \eqref{eq:chi_der2}, we have
\begin{align}
\begin{split}
&\quad
 \sum_{\substack{\Psi \text{ eigenfunction} \\ \text{of } \Delta \text{ with eigenvalue } < \lambda}} \int_{X \setminus X_{(k+1)B_\lambda}} |\Psi|^2 \, dv_g \\
&\leq \frac{1}{\lambda} \Tr_{L^{2}(X \setminus X_{kB_\lambda}, dv_{\tilde g})} \left( (2(1+\delta)^2 \lambda - \tilde{\Delta})_{+} \right) + \frac{1}{\lambda} \sum_{\substack{\Psi \text{ eigenfunction} \\ \text{of } \Delta \text{ with eigenvalue } < \lambda}} \int_{X} |\chi'|^2 |\Psi|^2 \, dv_g\\
&\leq \frac{1}{\lambda} \Tr_{L^{2}(X \setminus X_{kB_\lambda}, dv_{\tilde g})} \left( (2(1+\delta)^2 \lambda - \tilde{\Delta})_{+} \right) + \frac{4}{\lambda B_\lambda^2} \sum_{\substack{\Psi \text{ eigenfunction} \\ \text{of } \Delta \text{ with eigenvalue } < \lambda}} \int_{X_{(k+1)B_\lambda}\setminus X_{kB_\lambda}} |\Psi|^2 \, dv_g.
\end{split}
\end{align}
Now, dividing by $N(\lambda)$ and using the definition of $F_\lambda$, see \eqref{eq:F_lambda_def}, we obtain 
\begin{align}\label{eq:F_lambda_first_est_super_proof}
\begin{split}
&\quad\quad\int_{X\setminus X_{(k+1)B_\lambda}} F_\lambda\,dv_g= 
\frac{1}{N(\lambda)}\sum_{\substack{\Psi \text{ eigenfunction} \\ \text{of } \Delta \text{ with eigenvalue } < \lambda}} \int_{X\setminus X_{(k+1)B_\lambda}} |\Psi|^2 \, dv_g \\
&\leq \frac{1}{\lambda N(\lambda)}  \Tr_{L^{2}(X \setminus X_{kB_\lambda}, dv_{\tilde g})} \left( (2(1+\delta)^2 \lambda - \tilde{\Delta})_{+} \right) + \frac{4}{\lambda B_\lambda^2}  \int_{X_{(k+1)B_\lambda}\setminus X_{kB_\lambda}} F_\lambda\, dv_g.
\end{split}
\end{align}
By Proposition \ref{prop:number_eigenvalues}(i) applied with $B_\lambda$ replaced by $kB_\lambda$ and $\lambda$ replaced by $8\lambda$, we have
\begin{align}\label{eq:trace_proof_super}
\begin{split}
&\Tr_{L^{2}(X \setminus X_{kB_\lambda}, dv_{\tilde g})} \left( (2(1+\delta)^2 \lambda - \tilde{\Delta})_{+} \right)
\le 2(1+\delta)^2 \lambda N_{\tilde\Delta,X \setminus X_{kB_\lambda}}(2(1+\delta)^2 \lambda)\\
&\le 8 \lambda N_{\tilde\Delta,X \setminus X_{kB_\lambda}}(8 \lambda)
\le \lambda C \lambda^{\frac{n+1}{2}}(kB_{\lambda})^{-\frac{\beta n}{2}\left(1-\frac{2}{\beta n}\right)}=  C \lambda^{1+\frac{n+1}{2}}(kB_{\lambda})^{-\frac{\beta n}{2}\left(1-\frac{2}{\beta n}\right)}
.
\end{split}
\end{align}
Here the constant $C>0$ is independent of $k\in\N$.
Next, we combine 
\begin{equation}
N(\lambda)\sim C\lambda^{d/2}, \quad \text{as } \lambda\to\infty \quad \text{where } d:=n\left(1+\frac{\beta}{2}\right), 
\end{equation}
see \cite{yves}, and \eqref{eq:trace_proof_super}, to get from \eqref{eq:F_lambda_first_est_super_proof}
\begin{equation}\label{eq:before_bootstrap}
\int_{X\setminus X_{(k+1)B_\lambda}} F_\lambda\,dv_g
\le C \lambda^{\frac{1}{2}-\frac{\beta n}{4}}(kB_{\lambda})^{-\frac{\beta n}{2}\left(1-\frac{2}{\beta n}\right)}+ \frac{4}{\lambda B_\lambda^2}  \int_{X_{(k+1)B_\lambda}\setminus X_{kB_\lambda}} F_\lambda\, dv_g.
\end{equation}
At this point, we could estimate $\int_{X_{(k+1)B_\lambda}\setminus X_{kB_\lambda}} F_\lambda\, dv_g$ by $1$ and optimise in $B_\lambda$, pick $k=1$ and take $A_\lambda=B_\lambda/2$. However, we get a much better estimate if we use a bootstrap argument to show that for some $k\in\N$, the second term in \eqref{eq:before_bootstrap}, which represents the localisation error, has to be of the same order as the first term in the estimate in \eqref{eq:before_bootstrap}. 

\medskip

Let us now explain the details of the bootstrap argument. For $\gamma>0$, which we will choose as $\gamma:=1/(n\beta)$, we take 
\begin{equation}\label{eq:B_lambda_def}
B_\lambda := \frac{1}{4\gamma^{-1}+3}\lambda^{-\frac{1}{2}+\gamma}. 
\end{equation}
We will use \eqref{eq:before_bootstrap} for $k\in\N$ with $k\le 4\gamma^{-1}+2$. Using the notation
\begin{equation}
\F_k:= \int_{X\setminus X_{kB_\lambda}} F_\lambda\,dv_g, 
\end{equation}
\eqref{eq:before_bootstrap} implies for all $k\in\N$ with $k\le 4\gamma^{-1}+2$
\begin{equation}\label{eq:F_k1_est_ohne_gamma_gewaehlt}
\F_{k+1}\le C\lambda^{-\gamma\frac{\beta n}{2}\left(1-\frac{2}{\beta n}\right)}+C\lambda^{-2\gamma}\F_{k}
\end{equation}
with a constant $C>0$ that depends on $\gamma$, but which is independent of $k\in\N$ with $k\le 4\gamma^{-1}+2$. Since we later want to choose $A_\lambda$ as a constant times $B_\lambda$ and we aim for an estimate of the form \eqref{eq:goal_th_proof_recall_super}, the optimal choice for $\gamma$ is to take it as 
\begin{equation}\label{eq:gamma_def}
\gamma :=\frac{1}{n\beta}.
\end{equation}
With this choice, \eqref{eq:F_k1_est_ohne_gamma_gewaehlt} now reads
\begin{equation}\label{eq:F_k1_est}
\F_{k+1}\le C\lambda^{-\frac{1}{2}+\gamma}+C\lambda^{-2\gamma}\F_{k}
\end{equation}
We distinguish two cases:

{\bf Case 1: } There exists a $k\in\N$ with $k\le 4\gamma^{-1}+2$ such that 
\begin{equation}
\lambda^{-2\gamma}\F_{k}\le \lambda^{-\frac{1}{2}+\gamma}.
\end{equation}

{\bf Case 2: } For all $k\in\N$ with $k\le 4\gamma^{-1}+2$, we have
\begin{equation}
\lambda^{-2\gamma}\F_{k}\ge\lambda^{-\frac{1}{2}+\gamma}.
\end{equation}

\medskip

In case 1, we pick the corresponding $k_0$ and get from \eqref{eq:F_k1_est} 
\begin{equation}
\F_{k_0+1}\le 2C\lambda^{-\frac{1}{2}+\gamma}.
\end{equation}

\medskip

In case 2, we obtain from \eqref{eq:F_k1_est} that for all $k\in\N$ with $k\le 4\gamma^{-1}+2$, we have
\begin{equation}
\F_{k+1}\le 2C\lambda^{-2\gamma}\F_{k}.
\end{equation}
It follows that
\begin{equation}\label{eq:F_k1_est_case2}
\F_{k+1}\le \left(2C\lambda^{-2\gamma}\right)^{k-1}\F_{1}\le \left(2C\lambda^{-2\gamma}\right)^{k-1}\le C  \lambda^{-2\gamma(k-1)}
\end{equation}
where we used $\F_1\le1$, and the constant $C>0$ was possibly changed in the last step. Now we pick $k_0\in\N$ with $k_0\in[4\gamma^{-1}+1, 4\gamma^{-1}+2]$ and get from \eqref{eq:F_k1_est_case2} 
\begin{equation}
\F_{k_0+1}\le C \lambda^{-\frac{1}{2}}.
\end{equation}

\medskip

In both cases, we have shown that we can find $k_0\in\N$ with $k_0\le 4\gamma^{-1}+2$ and 
\begin{equation}\label{eq:F_k01_est}
\F_{k_0+1}\le 2C\lambda^{-\frac{1}{2}+\gamma}.
\end{equation}
Recall that we will take $A_\lambda=\lambda^{-\frac{1}{2}+\frac{1}{n\beta}}=\lambda^{-\frac{1}{2}+\gamma}$ and $B_\lambda$ was defined in \eqref{eq:B_lambda_def}. In particular, we have
\begin{equation}
(k_0+1)B_\lambda\le A_\lambda. 
\end{equation}
By \eqref{eq:F_k01_est}, it follows that 
\begin{equation}
\int_{X\setminus X_{A_\lambda}} F_\lambda\,dv_g\le \int_{X\setminus X_{(k_0+1)B_\lambda}} F_\lambda\,dv_g=\F_{k_0+1}\le 2C\lambda^{-\frac{1}{2}+\gamma}=2CA_\lambda. 
\end{equation}
This finishes the proof of \eqref{eq:goal_th_proof_recall_super} in the supercritical case $\beta>2/n$. 
\end{proof}

\section{Proof of Theorem \ref{th:moments}}\label{s:conclusion_moments}
In this section, we prove Theorem \ref{th:moments}. To this end, we will combine Proposition \ref{prop:localisation}(i) and \ref{prop:number_eigenvalues}(iii),(iv). 

\begin{proof}[Proof of Theorem \ref{th:moments}]
We take $\epsilon\in(0,1]$ small enough such that the corresponding $\delta>0$ satisfies $2(1+\delta)^2\le3$. We choose $\chi$ such that $\chi(x,y)=x^{p/2}$ for $(x,y)\in X_\epsilon$ and $\chi(z)=\epsilon^{p/2}$ for all $z\in X\setminus X_\epsilon$.
Dividing by $N(\lambda)$, Proposition \ref{prop:localisation}(i) yields
\begin{align}\label{eq:moments_proof_beginning}
\begin{split}
&\quad\int_{X} \min\{x^p,\epsilon^{p/2}\} F_\lambda(x,y)\,dv_g(x,y)\\
&\le \frac{1}{\lambda N(\lambda)} \Tr_{L^{2}(X, dv_{\tilde g})} \left( \chi(2(1+\delta)^2 \lambda - \tilde{\Delta})\chi\right)_{+}+ \frac{1}{\lambda}\int_{X} |\nabla\chi|^2 F_\lambda \, dv_g\\
&\le \frac{1}{\lambda N(\lambda)} \Tr_{L^{2}(X, dv_{\tilde g})} \left( \chi(2(1+\delta)^2 \lambda - \tilde{\Delta})\chi\right)_{+}+ \frac{1}{\lambda}.
\end{split}
\end{align}
Here we used that by $p\ge2$ and $\epsilon\le1$, we have $|\nabla\chi|=|\chi'|\le1$ and moreover, $\int_{X} F_\lambda \, dv_g=1$. In order to estimate the term with the trace, we will decompose the trace over ${L^{2}(X, dv_{\tilde g})}$ into traces over smaller sets: we have
\begin{equation}\label{eq:X_decomposition2}
X=(X\setminus X_\epsilon)\cup X_{2^{-1}\lambda^{-1/2}}\cup\bigcup_{k=0}^{\frac{2\log(\epsilon)+\log(\lambda)}{2\log(2)}}\left(X_{2^{k}\lambda^{-1/2}}\setminus X_{2^{k-1}\lambda^{-1/2}}\right), 
\end{equation} 
We have
\begin{align}\label{eq:trace_proof_split_moments}
\begin{split}
&\quad \Tr_{L^{2}(X, dv_{\tilde g})} \left( \chi(2(1+\delta)^2 \lambda - \tilde{\Delta})\chi\right)_{+}\\
&\le 
\Tr_{L^{2}(X\setminus X_\epsilon, dv_{\tilde g})} \left( \chi(2(1+\delta)^2 \lambda - \tilde{\Delta})\chi\right)_{+}
+\Tr_{L^{2}(X_{2^{-1}\lambda^{-1/2}}, dv_{\tilde g})} \left( \chi(2(1+\delta)^2 \lambda - \tilde{\Delta})\chi\right)_{+}
\\
&\quad
+\sum_{k=0}^{\frac{2\log(\epsilon)+\log(\lambda)}{2\log(2)}}\Tr_{L^{2}(X_{2^{k}\lambda^{-1/2}}\setminus X_{2^{k-1}\lambda^{-1/2}}, dv_{\tilde g})} \left( \chi(2(1+\delta)^2 \lambda - \tilde{\Delta})\chi\right)_{+}\\
&\le 2(1+\delta)^2 \lambda N_{\tilde\Delta, X\setminus X_\epsilon}(2(1+\delta)^2 \lambda)+0\\
&\quad
+\sum_{k=0}^{\frac{2\log(\epsilon)+\log(\lambda)}{2\log(2)}}\Tr_{L^{2}(X_{2^{k}\lambda^{-1/2}}\setminus X_{2^{k-1}\lambda^{-1/2}}, dv_{\tilde g})} \left( \chi(2(1+\delta)^2 \lambda - \tilde{\Delta})\chi\right)_{+}\\
&\le C\lambda^{1+\frac{n+1}{2}}+\sum_{k=0}^{\frac{2\log(\epsilon)+\log(\lambda)}{2\log(2)}}\Tr_{L^{2}(X_{2^{k}\lambda^{-1/2}}\setminus X_{2^{k-1}\lambda^{-1/2}}, dv_{\tilde g})} \left( \chi(2(1+\delta)^2 \lambda - \tilde{\Delta})\chi\right)_{+},
\end{split}
\end{align}
where we used Proposition \ref{prop:localisation}(iv) and $2(1+\delta)^2\le3$ for the second term. Note that for any set $0\le k \le \frac{2\log(\epsilon)+\log(\lambda)}{2\log(2)}$ and any $\Phi\in L^{2}(X_{2^{k}\lambda^{-1/2}}\setminus X_{2^{k-1}\lambda^{-1/2}}, dv_{\tilde g})$, and using 
\begin{equation}
(a+b)^2\ge\frac{1}{2}a^2-b^2
\end{equation}
and $\chi'(x)=\frac{p}{2} x^{\frac{p}{2}-1}$, we have
\begin{align}
\begin{split}
&\quad\int_{X_{2^{k}\lambda^{-1/2}}\setminus X_{2^{k-1}\lambda^{-1/2}}}|\tilde\nabla(\chi\Phi)|^2\, dv_{\tilde g} 
\ge  \int_{X_{2^{k}\lambda^{-1/2}}\setminus X_{2^{k-1}\lambda^{-1/2}}}|(\tilde\nabla\chi)\Phi+\chi\tilde\nabla\Phi|^2\, dv_{\tilde g} \\
&\ge \int_{X_{2^{k}\lambda^{-1/2}}\setminus X_{2^{k-1}\lambda^{-1/2}}}\left(\frac{1}{2}|\chi\tilde\nabla\Phi|^2-|(\tilde\nabla\chi)\Phi|^2\right)\, dv_{\tilde g}\\
&\ge \int_{X_{2^{k}\lambda^{-1/2}}\setminus X_{2^{k-1}\lambda^{-1/2}}}\left(2^{p(k-1)-1}\lambda^{-p/2}|\tilde\nabla\Phi|^2-\frac{p^2}{4}2^{kp-2}\lambda^{-\frac{p}{2}+1}|\Phi|^2\right)\, dv_{\tilde g}.
\end{split}
\end{align}
It follows that
\begin{align}
\begin{split}
&\quad \Tr_{L^{2}(X_{2^{k}\lambda^{-1/2}}\setminus X_{2^{k-1}\lambda^{-1/2}}, dv_{\tilde g})} \left( \chi(2(1+\delta)^2 \lambda - \tilde{\Delta})\chi\right)_{+}\\
&\le 
\Tr_{L^{2}(X_{2^{k}\lambda^{-1/2}}\setminus X_{2^{k-1}\lambda^{-1/2}}, dv_{\tilde g})} \left( 2(1+\delta)^2 \lambda 2^{pk}\lambda^{-p/2}+\frac{p^2}{4}2^{kp-2}\lambda^{-\frac{p}{2}+1}- 2^{p(k-1)-1}\lambda^{-p/2}\tilde{\Delta}\right)_{+}\\
&=2^{p(k-1)-1}\lambda^{-p/2}\Tr_{L^{2}(X_{2^{k}\lambda^{-1/2}}\setminus X_{2^{k-1}\lambda^{-1/2}}, dv_{\tilde g})} \left( 2^{p+2}(1+\delta)^2 \lambda +p^2 2^{p-2}\lambda- \tilde{\Delta}\right)_{+}\\
&\le C_{p,\delta}2^{p(k-1)}\lambda^{1-p/2}N_{\tilde\Delta, X_{2^{k}\lambda^{-1/2}}\setminus X_{2^{k-1}\lambda^{-1/2}}}(C_{p,\delta} \lambda)
\end{split}
\end{align}
Here we denote by $C_{p,\delta} $ a constant that only depends on $p$ and $\delta$, which may change from line to line. Next, we use Proposition \ref{prop:number_eigenvalues}(iii) with $B_\lambda=2^{(k-1)}\lambda^{-1/2}$ to get
\begin{equation}
N_{\tilde\Delta, X_{2^{k}\lambda^{-1/2}}\setminus X_{2^{k-1}\lambda^{-1/2}}}(C_{p,\delta} \lambda)\le C_{p,\delta,\beta}\lambda^{\frac{n+1}{2}}(2^{k-1}\lambda^{-1/2})^{-\frac{\beta n}{2}\left(1-\frac{2}{\beta n}\right)}
\end{equation}
It follows that
\begin{equation}\label{eq:individual_trace_est_moments}
\Tr_{L^{2}(X_{2^{k}\lambda^{-1/2}}\setminus X_{2^{k-1}\lambda^{-1/2}}, dv_{\tilde g})} \left( \chi(2(1+\delta)^2 \lambda - \tilde{\Delta})\chi\right)_{+}\le 
C_{p,\delta,\beta,n}\lambda^{1+\frac{n+1}{2}} \int_{2^{k-1}\lambda^{-1/2}}^{2^{k}\lambda^{-1/2}}s^{p-\frac{\beta n}{2}}\, ds. 
\end{equation}
Combining \eqref{eq:trace_proof_split_moments} and \eqref{eq:individual_trace_est_moments}, we obtain
\begin{equation}
\Tr_{L^{2}(X, dv_{\tilde g})} \left( \chi(2(1+\delta)^2 \lambda - \tilde{\Delta})\chi\right)_{+}\le C\lambda^{1+\frac{n+1}{2}}+C_{p,\delta,\beta,n}\lambda^{1+\frac{n+1}{2}} \int_{\frac{1}{2}\lambda^{-1/2}}^{\epsilon}s^{p-\frac{\beta n}{2}}\, ds. 
\end{equation}
For $\lambda$ large enough, the second term is greater than or equal to the first one up to a constant. Now plugging this into \eqref{eq:moments_proof_beginning} and observing that we  have $N(\lambda)\ge C\lambda^{\frac{n+1}{2}}$, we obtain for $\lambda$ large enough,
\begin{equation}\label{eq:moments_proof_int_p_beta}
\int_{X} \min\{x^p,\epsilon^{p/2}\} F_\lambda(x,y)\,dv_g(x,y)
\le C_{p,\delta,\beta,n,\epsilon}\lambda^{\frac{n+1}{2}}\frac{1}{N(\lambda)}\int_{\frac{1}{2}\lambda^{-1/2}}^{\epsilon}s^{p-\frac{\beta n}{2}}\, ds+\frac{1}{\lambda}. 
\end{equation}
If $\beta=2/n$, then since $p\ge2>0$, we can always estimate
\begin{equation}
\int_{\frac{1}{2}\lambda^{-1/2}}^{\epsilon}s^{p-\frac{\beta n}{2}}\, ds\le C_{\epsilon, p,n,\beta}.
\end{equation}
Moreover, in this case, we have
\begin{equation}
N(\lambda)\sim C\lambda^{\frac{n+1}{2}}\log(\lambda),
\end{equation}
Thus, from \eqref{eq:moments_proof_int_p_beta}, we get
\begin{equation}
\int_{X} \min\{x^p,\epsilon^{p/2}\} F_\lambda(x,y)\,dv_g(x,y)
\le C_{p,\delta,\beta,n,\epsilon}\frac{1}{\log(\lambda)}+\frac{1}{\lambda}\le C_{p,\delta,\beta,n,\epsilon}\frac{1}{\log(\lambda)}. 
\end{equation}

\medskip

Next, let us consider the case where $\beta>2/n$. In this case, 
\begin{equation}
N(\lambda)\sim C\lambda^{\frac{d}{2}}
\end{equation} 
and by $d=n(1+\beta/2)$, we find that
\begin{equation}
\lambda^{\frac{n+1}{2}}\frac{1}{N(\lambda)}\sim C\lambda^{\frac{1}{2}-\frac{\beta n}{4}}.
\end{equation}
Let us distinguish three cases: 

{\bf Case 1: $p>-1+\frac{\beta n}{2}$. } 
We can estimate
\begin{equation}
\int_{\frac{1}{2}\lambda^{-1/2}}^{\epsilon}s^{p-\frac{\beta n}{2}}\, ds\le C_{\epsilon, p,n,\beta}.
\end{equation}
Thus, from \eqref{eq:moments_proof_int_p_beta}, we get
\begin{equation}
\int_{X} \min\{x^p,\epsilon^{p/2}\} F_\lambda(x,y)\,dv_g(x,y)
\le C_{p,\delta,\beta,n,\epsilon}\lambda^{\frac{1}{2}-\frac{\beta n}{4}}+\frac{1}{\lambda}\le C_{p,\delta,\beta,n,\epsilon}\lambda^{\max\left\{\frac{1}{2}-\frac{\beta n}{4},-1\right\}}.
\end{equation}

{\bf Case 2: $p=-1+\frac{\beta n}{2}$. } In this case, we have 
\begin{equation}
\int_{\frac{1}{2}\lambda^{-1/2}}^{\epsilon}s^{p-\frac{\beta n}{2}}\, ds\le C_{\epsilon, p,n,\beta}\log(\lambda)
\end{equation}
for $\lambda$ large enough. Moreover, by $p\ge2$,
\begin{equation}
\frac{1}{2}-\frac{\beta n}{4}=-\frac{p}{2}\le -1
\end{equation}
and therefore we get  
\begin{equation}
\int_{X} \min\{x^p,\epsilon^{p/2}\} F_\lambda(x,y)\,dv_g(x,y)
\le 
\begin{cases}
C_{p,\delta,\beta,n,\epsilon} \frac{\log(\lambda)}{\lambda} & \text{if } p=2 \\
C_{p,\delta,\beta,n,\epsilon} \frac{1}{\lambda}       & \text{if } p>2
\end{cases}
\end{equation}

{\bf Case 3: $p<-1+\frac{\beta n}{2}$. } In this case, we have
\begin{equation}
\int_{\frac{1}{2}\lambda^{-1/2}}^{\epsilon}s^{p-\frac{\beta n}{2}}\, ds\le C_{\epsilon, p,n,\beta}\lambda^{-\frac{1}{2}\left(1+p-\frac{\beta n}{2}\right)}=C_{\epsilon, p,n,\beta}\lambda^{\frac{1}{2}\left(\frac{\beta n}{2}-1-p\right)}
\end{equation}
Note that
\begin{equation}
C_{p,\delta,\beta,n,\epsilon}\lambda^{\frac{n+1}{2}}\frac{1}{N(\lambda)}\int_{\frac{1}{2}\lambda^{-1/2}}^{\epsilon}s^{p-\frac{\beta n}{2}}\, ds
\le C_{p,\delta,\beta,n,\epsilon}\lambda^{\frac{1}{2}-\frac{\beta n}{4}}\lambda^{\frac{1}{2}\left(\frac{\beta n}{2}-1-p\right)}=C_{p,\delta,\beta,n,\epsilon}\lambda^{-\frac{p}{2}}.
\end{equation}
By $p\ge2$, it follows that
\begin{equation}
\int_{X} \min\{x^p,\epsilon^{p/2}\} F_\lambda(x,y)\,dv_g(x,y)
\le C_{p,\delta,\beta,n,\epsilon} \frac{1}{\lambda}.
\end{equation}

\bigskip 

{\bf Summary. }
To sum up, if $\beta=2/n$, then we have
\begin{equation}
\int_{X} \min\{x^p,\epsilon^{p/2}\} F_\lambda(x,y)\,dv_g(x,y)
\le C_{p,\delta,\beta,n,\epsilon}\frac{1}{\log(\lambda)}. 
\end{equation}
If $\beta>2/n$, and
if $p=2$ and $\beta=6/n$, then we have 
\begin{equation}
\int_{X} \min\{x^p,\epsilon^{p/2}\} F_\lambda(x,y)\,dv_g(x,y)
\le C_{p,\delta,\beta,n,\epsilon} \frac{\log(\lambda)}{\lambda}.
\end{equation}
In all other cases for $\beta>2/n$, we get
\begin{equation}
\int_{X} \min\{x^p,\epsilon^{p/2}\} F_\lambda(x,y)\,dv_g(x,y)
\le C_{p,\delta,\beta,n,\epsilon} \lambda^{\max\left\{\frac{1}{2}-\frac{\beta n}{4},-1\right\}}.
\end{equation}
Dividing the left-hand side by $\epsilon^{p/2}$, we obtain the desired result. 

\bigskip

In view of Remark \ref{re:optimality} below, also note that in the following cases, our estimate for the localisation term is of subleading order or at most of the same order as the main term:
\begin{itemize}
\item $\beta=2/n$
\item $\beta>2/n$ and $p>-1+\frac{\beta n}{2}$ and $\frac{1}{2}-\frac{\beta n}{4}\ge -1$
\item $\beta>2/n$ and $p=-1+\frac{\beta n}{2}$ and $p=2$
\end{itemize}
Put differently, since $p\ge 2$ and $\beta\ge2/n$, we can also express these conditions as 
\begin{equation}\label{eq:beta_cond}
\frac{2}{n}\le \beta\le\frac{6}{n}.
\end{equation}
\end{proof}

\begin{remark}\label{re:optimality}
Both in \eqref{eq:X_ohne_X_A_lambda_into_th} and for $p\ge2$, $\frac{2}{n}\le\beta\le\frac{6}{n}$ in \eqref{eq:moments theorem}, we cannot hope for a much better estimate.
More precisely, in the setting of \eqref{eq:X_ohne_X_A_lambda_into_th} in  Theorem \ref{th:main}, for every $\lambda$ large enough, there exists an orthonormal family in $L^2(X,dv_g)$ of at least $\frac{1}{2}N(\lambda)$ functions, which we call $\{\hat\Psi\}$, such that for every $\hat\Psi$
\begin{equation}\label{eq:ev_lambda_main}
\int_X|\nabla\hat\Psi|^2\,dv_g<\lambda
\end{equation}
and moreover, if $\hat A_\lambda\ge C\lambda^{-1/2}$ for some $C>0$ large enough is such that
\begin{equation}
\frac{1}{N(\lambda)}\int_{X\setminus X_{\hat A_\lambda}} |\hat\Psi|^2\,dv_g\le C \hat A_\lambda,
\end{equation}
then for any $\epsilon>0$, we have for $\lambda$ large enough, 
\begin{equation}
\hat A_\lambda\ge  A_\lambda \lambda^{-\epsilon}.
\end{equation}

\bigskip

Similarly, in the setting of \eqref{eq:moments theorem} in Theorem \ref{th:moments} if $p\ge2$ and $\frac{2}{n}\le\beta\le\frac{6}{n}$, then
%
%
%
%
%
for every $\lambda$ large enough, there exists an orthonormal family in $L^2(X,dv_g)$ of at least $\frac{1}{2}N(\lambda)$ functions, which we call $\{\hat\Psi\}$, such that for every $\hat\Psi$
\begin{equation}
\int_X|\nabla\hat\Psi|^2\,dv_g<\lambda
\end{equation}
and such that we have 
\begin{equation}\label{eq:moments_opt}
\frac{1}{N(\lambda)}\int_{X} \min\{x^p,1\}|\hat\Psi(x,y)|^2\,dv_g(x,y)\ge C A_\lambda\lambda^{-\epsilon}. 
\end{equation}
\end{remark}

\begin{proof}[Proof of Remark \ref{re:optimality}]
In order to see the statement of Remark \ref{re:optimality}, it suffices to take the corresponding set of eigenfunctions of the Dirichlet Laplacian and use the last part of Proposition \ref{prop:number_eigenvalues}. We give a sketch of the proof below. 

\medskip

More precisely, for the optimality of \eqref{eq:X_ohne_X_A_lambda_into_th} in Theorem \ref{th:main}, we take the union of the normalised eigenfunctions of $\Delta$ with Dirichlet boundary conditions at $\partial X_{A_\lambda}$  with eigenvalue less than $\lambda$ on $X_{A_\lambda}$ and on $X\setminus X_{A_\lambda}$. For $\lambda$ large enough, these are at least $\frac{1}{2}N(\lambda)$ eigenfunctions and they satisfy \eqref{eq:ev_lambda_main}, and using Proposition \ref{prop:number_eigenvalues}(v), (vi) and the quasi-isometry of $g$ and $\tilde g$, we obtain the desired result in this case. 

\medskip

For the optimality of \eqref{eq:moments theorem} in Theorem \ref{th:moments} for $p\ge2$ and $\frac{2}{n}\le\beta\le\frac{6}{n}$, one can proceed similarly. We decompose $X$ as follows:
\begin{equation}\label{eq:X_decomposition}
X=(X\setminus X_\epsilon)\cup X_{2^{-1}\lambda^{-1/2}}\cup\bigcup_{k=0}^{\frac{2\log(\epsilon)+\log(\lambda)}{2\log(2)}}\left(X_{2^{k}\lambda^{-1/2}}\setminus X_{2^{k-1}\lambda^{-1/2}}\right), 
\end{equation} 
Then we take the set of eigenfunctions of $\Delta$ with Dirichlet boundary conditions on each of those smaller pieces, use the quasi-isometry of $g$ and $\tilde g$, and Proposition \ref{prop:number_eigenvalues}(vii). Then one can follow the same computations as at the end of the proof of Theorem \ref{th:moments}.
\end{proof}

\bibliographystyle{plain}
\bibliography{References.bib}

\begin{thebibliography}{10}

\bibitem{abatangelo2025solutions}
Laura Abatangelo, Alberto Ferrero, and Paolo Luzzini.
\newblock On solutions to a class of degenerate equations with the grushin
  operator.
\newblock {\em Journal of Differential Equations}, 445:113666, 2025.

\bibitem{boscain}
Ugo Boscain, Dario Prandi, and Marcello Seri.
\newblock Spectral analysis and the aharonov-bohm effect on certain
  almost-riemannian manifolds.
\newblock {\em Communications in Partial Differential Equations}, 41(1):32--50,
  2016.

\bibitem{chang2015heat}
Der-Chen Chang and Yutian Li.
\newblock Heat kernel asymptotic expansions for the heisenberg sub-laplacian
  and the grushin operator.
\newblock {\em Proceedings of the Royal Society A: Mathematical, Physical and
  Engineering Sciences}, 471(2175):20140943, 2015.

\bibitem{chitour}
Yacine Chitour, Dario Prandi, and Luca Rizzi.
\newblock Weyl's law for singular riemannian manifolds.
\newblock {\em Journal de Math{\'e}matiques Pures et Appliqu{\'e}es},
  181:113--151, 2024.

\bibitem{yves}
Yves {Colin de Verdi{\`e}re}, Charlotte Dietze, Maarten~V de~Hoop, and Emmanuel
  Tr{\'e}lat.
\newblock Weyl formulae for some singular metrics with application to acoustic
  modes in gas giants.
\newblock {\em arXiv preprint arXiv:2406.19734}, 2024.

\bibitem{colin2018spectral}
Yves {Colin de Verdi{\`e}re}, Luc Hillairet, and Emmanuel Tr{\'e}lat.
\newblock Spectral asymptotics for sub-riemannian laplacians, i: Quantum
  ergodicity and quantum limits in the 3-dimensional contact case.
\newblock {\em Duke Math. J.}, 167:109--174, 2018.

\bibitem{colin2021small}
Yves {Colin de Verdi{\`e}re}, Luc Hillairet, and Emmanuel Tr{\'e}lat.
\newblock Small-time asymptotics of hypoelliptic heat kernels near the
  diagonal, nilpotentization and related results.
\newblock {\em Annales Henri Lebesgue}, 4:897--971, 2021.

\bibitem{colin2022spectral}
Yves {Colin de Verdi{\`e}re}, Luc Hillairet, and Emmanuel Tr{\'e}lat.
\newblock Spectral asymptotics for sub-riemannian laplacians.
\newblock {\em arXiv preprint arXiv:2212.02920}, 2022.

\bibitem{charlotte}
Charlotte Dietze.
\newblock The critical case for the concentration of eigenfunctions on singular
  riemannian manifolds.
\newblock {\em arXiv preprint arXiv:2510.23520}, 2025.

\bibitem{PhD}
Charlotte Dietze.
\newblock {\em Spectral estimates for singular systems}.
\newblock PhD thesis, LMU Munich, 2025.

\bibitem{larry}
Charlotte Dietze and Larry Read.
\newblock Concentration of eigenfunctions on singular riemannian manifolds.
\newblock {\em arXiv preprint arXiv:2410.20563}, 2024.

\bibitem{hoermander3}
Lars H{\"o}rmander.
\newblock {\em The analysis of linear partial differential operators III:
  Pseudo-differential operators}.
\newblock Springer Science \& Business Media, 2007.

\bibitem{lamberti2021shape}
Pier~Domenico Lamberti, Paolo Luzzini, and Paolo Musolino.
\newblock Shape perturbation of grushin eigenvalues.
\newblock {\em The Journal of Geometric Analysis}, 31(11):10679--10717, 2021.

\bibitem{menikoff1978eigenvalues}
Arthur Menikoff and Johannes Sj{\"o}strand.
\newblock On the eigenvalues of a class of hypoelliptic operators.
\newblock {\em Mathematische Annalen}, 235(1):55--85, 1978.

\bibitem{metivier}
Guy M{\'e}tivier.
\newblock Comportement asymptotique des valeurs propres d'op{\'e}rateurs
  elliptiques d{\'e}g{\'e}n{\'e}r{\'e}s.
\newblock {\em Journ{\'e}es {\'e}quations aux d{\'e}riv{\'e}es partielles},
  pages 215--249, 1975.

\bibitem{vulis}
IL~Vulis and MZ~Solomjak.
\newblock Spectral asymptotic analysis for degenerate second order elliptic
  operators.
\newblock {\em Izv. Akad. Nauk SSSR Ser. Mat}, 38:1362--1392, 1974.

\end{thebibliography}

\end{document}